\documentclass[smallextended,envcountsect]{svjour3}
\usepackage{pifont}
\usepackage{epsfig,color,graphicx,subfigure,amsmath,amssymb,multirow,lscape}
\usepackage[colorlinks,linkcolor=blue,citecolor=blue]{hyperref}
\usepackage{tikz,cite}
\usetikzlibrary{arrows, automata}
\smartqed
\usepackage{graphicx}
\usepackage{amsmath,amssymb}
\usepackage{appendix}

\journalname{}

\newtheorem{thm}{Theorem}[section]
\newtheorem{lem}{Lemma}[section]

\newtheorem{rem}{Remark}[section]
\newtheorem{fact}{Fact}[section]

\newtheorem{algo}{Algorithm}[section]
\newtheorem{prob}{Problem}[section]
\newtheorem{asmp}{Assumption}[section]

\def\be{\begin{eqnarray}}
\def\ee{\end{eqnarray}}
\def\ben{\begin{eqnarray*}}
\def\een{\end{eqnarray*}}
\def\ba{\begin{array}}
\def\ea{\end{array}}
\def\bi{\begin{itemize}}
\def\ei{\end{itemize}}

\def\cL{{\mathcal L}}

\def\cN{{\mathcal N}}
\def\cO{{\mathcal O}}

\def\bR{{\mathbb R}}

\def\prox{{\rm Prox}}
\def\null{{\rm null}}

\def\[{\begin{equation}}
\def\]{\end{equation}}

\newcommand{\D}{\Delta}

\newcommand{\R}{\mathbb R}

\newcommand{\lr}[1]{\left\langle #1\right\rangle}

\begin{document}

\title{Golden ratio primal-dual algorithm with linesearch}
\titlerunning{Golden ratio primal-dual algorithm with linesearch}
\author{Xiaokai Chang$^{1}$
\and  Junfeng Yang$^2$
\and Hongchao Zhang$^3$
}

\institute{1 \ \   School of Science, Lanzhou University of Technology, Lanzhou, Gansu, P. R. China. Email: xkchang@lut.cn. \\
2  \ \  Department of Mathematics, Nanjing University, Nanjing, P. R. China. Email: jfyang@nju.edu.cn. \\
3 \ \    Department of Mathematics,
Louisiana State University, Baton Rouge, LA 70803-4918. Phone (225) 578-1982. Fax (225) 578-4276.
Email: hozhang@math.lsu.edu. http://www.math.lsu.edu/$\sim$hozhang.
}

\date{Received: date / Accepted: date}

\maketitle

\begin{abstract}
Golden ratio primal-dual algorithm (GRPDA) is a new variant of the classical Arrow-Hurwicz method for solving structured convex optimization problem, in which the objective function consists of the sum of two closed proper convex functions, one of which involves a composition with a linear transform.
The same as the Arrow-Hurwicz method and the popular primal-dual algorithm (PDA) of Chambolle and Pock, GRPDA is full-splitting in the sense that it does not rely on solving any subproblems or linear system of equations iteratively.
Compared with PDA, an important feature of GRPDA is that it permits larger primal and dual stepsizes.
However, the stepsize condition of standard GRPDA requires that the spectral norm of the linear transform is known, which can be difficult to obtain in some applications. Furthermore,  constant stepsizes prescribed by the stepsize condition are usually overconservative in practice.

In this paper, we propose a linesearch strategy for GRPDA, which not only does not require the spectral norm of the linear transform  but also allows adaptive and  potentially much larger stepsizes. Within each linesearch step, only the dual variable needs to be updated, and it is thus quite cheap and does not require any extra matrix-vector multiplications for many special yet important applications, e.g., regularized least squares problem. Global convergence and ${\cal O}(1/N)$ ergodic convergence rate results measured by the primal-dual gap function are established, where $N$ denotes the iteration counter. When one of the component functions is strongly convex, faster ${\cal O}(1/N^2)$ ergodic convergence rate results are established by adaptively choosing some algorithmic parameters. Moreover, when both component functions are strongly convex, nonergodic linear converge results are established.
Numerical experiments on matrix game and LASSO problems illustrate the effectiveness of the proposed linesearch strategy.
\end{abstract}

\section{Introduction}
\label{sec:intro}
Let $\R^p$ and $\R^q$ be finite-dimensional Euclidean spaces, each endowed with an inner product and the induced norm denoted by $\langle\cdot, \cdot\rangle$ and $\|\cdot\| =\sqrt{\langle\cdot,\cdot\rangle}$, respectively. Let $f :\R^p \rightarrow(-\infty, +\infty]$ and $g :\R^q \rightarrow(-\infty, +\infty]$ be extended real-valued closed proper convex functions, $K\in\R^{p\times q}$ be a linear transform from $\R^q$ to $\R^p$. Denote the Legendre-Fenchel conjugate of $f$ by $f^*$, i.e., $f^*(y) = \sup_{x\in\R^p}\{ \langle y, x\rangle - f(x)\}$, $y\in\R^p$.
In this paper, we focus on the following saddle point problem with a bilinear coupling term
\be\label{saddle_point}
\min_{x\in \R^q}\max_{y\in \R^p} ~~\mathcal{L}(x,y):=g(x)+\langle Kx,y\rangle-f^*(y).
\ee
Since the biconjugate of $f$ is itself, i.e., $(f^*)^* = f$, see  \cite{Rockafellar1970Convex},   problem (\ref{saddle_point}) reduces to the following primal minimization problem
\be\label{primal}
\min_{x\in \R^q} g(x) + f(Kx).
\ee
On the other hand, by swapping the ``min" and the ``max",  problem \eqref{saddle_point} can be transformed to the following dual maximization problem
\be\label{dual}
\max_{y\in \R^p} -f^*(y) - g^*(-K^\top y),
\ee
where $K^\top$ denotes the matrix transpose or adjoint operator of $K$. Under regularity conditions, e.g., Assumption \ref{asmp-1} given below, strong duality holds between \eqref{primal} and \eqref{dual}.

Problems \eqref{saddle_point}-\eqref{dual} naturally arise from abundant interesting applications, including signal and image processing, machine learning, statistics, mechanics and economics, and so on, see, e.g., \cite{Chambolle2011A,Bouwmans2016Handbook,Yang2011Alternating,Hayden2013A,Bertsekas1982Projection}
and the references therein.
To solve \eqref{saddle_point}-\eqref{dual} simultaneously, popular choices include the well known alternating direction method of multipliers (ADMM) \cite{GM75,GabM76},
the primal-dual algorithm (PDA) of Chambolle and Pock \cite{Chambolle2011A,He2012Convergence,Pock2011Diagonal}, and their accelerated and generalized variants \cite{Liu2018Acceleration,Malitsky2018A}.
The focus of this paper is  primal-dual type full-splitting algorithms\footnote{An algorithm is said to be full-splitting if it does not rely on solving any subproblems or linear system of equations iteratively and the main computations per iteration are matrix-vector multiplications and the evaluations of proximal operators.}  for solving \eqref{saddle_point}-\eqref{dual}.
We emphasize that the literature on numerical algorithms for solving \eqref{saddle_point}-\eqref{dual} has become fairly vast and a thorough overview is not only impossible but also far beyond the focus of this work.
Instead, we review only some primal-dual type algorithms that are most closely related to this work.
Before going into details, we define our notation.

\subsection{Notations}
%Our notation is quite standard in convex analysis and optimization. %\cite{Rockafellar1970Convex}.
%For clearness, we recall some notation that will be used extensively throughout.

As already mentioned above, the transpose operation of a matrix or a vector is denoted by superscript ``$\top$".
The spectral norm of $K$ is denoted by $L$, i.e., $L := \|K\| = \sup \{\langle Kx, y\rangle: \|x\|=\|y\|=1, x\in\R^q, y\in\R^p\}$.
Let $h$ be any extended real-valued closed proper convex function defined on a finite dimensional Euclidean space $\R^m$.
The effective domain of $h$ is denoted by $\text{dom}(h) := \{ x\in\R^m: h(x) < +\infty\}$, and the subdifferential of $h$ at $x\in \R^m$ is denoted by $\partial h(x) := \{\xi\in\R^m: \, h(y) \geq h(x) + \langle \xi, y-x\rangle \text{~for all~}y\in\R^m\}$. Furthermore, for $\lambda >0$, the proximal operator  of $\lambda h$ is given by
\begin{eqnarray*}\label{def:prox}
  \prox_{\lambda h}(x) := \arg\min_{y\in \R^m } \Big\{h(y) + {1\over 2\lambda }\|y-x\|^2\Big\}, \quad x\in \R^m.
\end{eqnarray*}
%The indicator function of a set $C$ is denoted by $\iota_C(x)$, i.e., $\iota_C(x) = 0$ if $x\in C$ and $+\infty$ if otherwise.
The relative interior of $C$ is denoted by $\text{ri}(C)$.
%The identity operator or identity matrix is denoted by $I$, whose domain or order is clear from the context.
%The zero vector or matrix is simply denoted by $0$.
%The composition of two operators is denoted by ``$\circ$".
%The stagnation of two column vectors $u$ and $v$ is also denoted by $(u; v)$, i.e., $(u; v) = (u^\top, v^\top)^\top$.
%The sequence of nonnegative natural numbers is denoted by $\bN = \{0, 1,2,\ldots\}$.
Finally, throughout this paper, we denote the golden ratio by $\phi$, i.e., $\phi = {\sqrt{5}+1 \over 2}$, which  is a key parameter in golden ratio type algorithms.
Other notations will be specified in the context.

\subsection{Related works}

The theme of this paper is to incorporate linesearch into the golden ratio primal-dual algorithm, which was
originally introduced in \cite{ChY2020Golden} for solving \eqref{saddle_point}-\eqref{dual}.
A main feature of primal-dual type algorithms is that the three problems  \eqref{saddle_point}-\eqref{dual} are solved simultaneously by alternatingly updating the primal and the dual variables. Among others, the classical augmented Lagrangian method and its variants such as ADMM \cite{GM75,GabM76,Jonathan1992On,Lions1979Splitting} are most popular. However, ADMM is not full-splitting since at each iteration it requires to solve a subproblem of the form
$\min_{x\in \R^q}  \frac{1}{2}\|Kx-b_n\|^2 + g(x)$, where $b_n\in \R^p$ varies with the iteration counter $n$. Note that even if the proximal operator of $g$ is easy to evaluate, this problem needs to be solved iteratively, unless $K$ is the identity operator. On the other hand, for regularized least-squares problem, ADMM requires to solve a linear system of equations at each iteration, which could be prohibitive for large scale applications.

The most classical and simple full-splitting algorithm designed in the literature for solving \eqref{saddle_point}-\eqref{dual} goes back to \cite{Uzawa58}, which is nowadays widely known as the Arrow-Hurwicz method.  Started at $x_0\in\R^q$ and $y_0\in\R^p$, the Arrow-Hurwicz method iterates for $n\geq 1$ as
\ben %\label{pda_AH}
\left\{
\ba{l}
x_{n}=\prox_{\tau g}(x_{n-1}-\tau K^\top y_{n-1}),\\
y_{n}=\prox_{\sigma f^*}(y_{n-1}+\sigma K x_{n}),
\ea\right.
\een
where $\tau, \sigma > 0$ are stepsize parameters.
Since $f$ and $g$, as well as their conjugate functions, are closed proper and convex,  their proximal operators are  uniquely well defined everywhere.
The heuristics of Arrow-Hurwicz method %\eqref{pda_AH}
is to solve the minimax problem \eqref{saddle_point} by alternatingly minimizing with the primal variable $x$ and maximizing with the dual variable $y$ and, meanwhile, incorporating proximal steps by taking into account the latest information of each of the variables.
Convergence of the Arrow-Hurwicz method with small stepsizes was studied in \cite{Esser2010General}, and a sublinear convergence rate result, measured by primal-dual gap function, was obtained in
\cite{Chambolle2011A,Nedic2009Subgradient} when $\text{dom}(f^*)$ is bounded.
Though intuitively make sense,  the Arrow-Hurwicz method does not converge in general. In fact, a divergent example has been constructed in \cite{He2014On}. Nonetheless, this method has been popular in image processing community and is known as primal-dual hybrid gradient method \cite{ZhC08cam,Esser2010General,Chambolle2011A}.

To obtain a convergent full-splitting algorithm under more general setting, Chambolle and Pock \cite{Chambolle2011A,Chambolle2016ergodic} modified the Arrow-Hurwicz method by adopting an extrapolation step. Specifically, $x_n$ is replaced by the extrapolated point $z_n := x_{n}+\delta (x_{n}-x_{n-1})$ in the computation of $y_n$, where $\delta \in (0,1]$ is an extrapolation parameter, resulting
the following iterative scheme
\be\label{pda_CP}
\left\{
\ba{l}
x_{n}=\prox_{\tau g}(x_{n-1}-\tau K^\top y_{n-1}),\\
z_{n}=x_{n}+\delta (x_{n}-x_{n-1}),\\
y_{n}=\prox_{\sigma f^*}(y_{n-1}+\sigma K z_{n}).
\ea\right.
\ee
The convergence of (\ref{pda_CP}) with $\delta=1$ was established in \cite{Chambolle2011A} under the condition $\tau\sigma L^2 < 1$. Later, it was shown in \cite{He2012Convergence} that the iterative scheme \eqref{pda_CP}, with $\delta=1$, is an application of a weighted proximal point algorithm to solve an
equivalent mixed variational inequality of the optimality conditions of \eqref{saddle_point}. Furthermore, the scheme \eqref{pda_CP}, again with $\delta = 1$, is also referred to as split inexact Uzawa method in \cite{Esser2010General}, where the connection of PDA with preconditioned or linearized ADMM has been revealed, see \cite{Chambolle2011A,Shefi2014Rate}.
Note that, without taking a correction step as done in \cite{He2012Convergence}, the convergence of \eqref{pda_CP} with $\delta\in (0,1)$ is still open. The overrelaxed, inertial and accelerated versions of (\ref{pda_CP}) were investigated in \cite{Chambolle2016ergodic}, and its stochastic extension was studied in \cite{Chambolle2018STOCHASTIC}.

Recently, based on a seminal convex combination technique originally introduced by Malitsky \cite{Malitsky2019Golden} for variational inequality problem, we proposed a golden ratio primal-dual algorithm (GRPDA) in \cite{ChY2020Golden}.
Instead of an extrapolation step as taken in \eqref{pda_CP}, a convex combination of essentially all the primal iterates generated till far, i.e., $\{x_i: i = 0, 1, \ldots, n - 1\}$, is used in the $n$-th iteration of GRPDA. Specifically, given $x_0\in\R^q$ and let $z_0 := x_0$, GRPDA iterates for $n\geq 1$ as
\be\label{GRPDA}
\left\{\ba{rcl}
 z_{n}&=&\frac{\psi-1}{\psi} x_{n-1} + \frac{1}{\psi}z_{n-1}, \smallskip \\
 x_{n}&=&\prox_{\tau g}(z_{n}-\tau K^\top y_{n-1}), \smallskip \\
 y_{n}&=&\prox_{\sigma f^*}(y_{n-1}+\sigma K x_{n}).
 \ea\right.\ee
Here $\psi\in (1, \phi]$ determines the convex combination coefficients.
Global iterative convergence and ergodic convergence rate results are established  under the condition $\tau\sigma L^2<\psi$ in \cite{ChY2020Golden}. Since $\psi\in (1,\phi]$ and $\phi = {\sqrt{5}+1 \over 2}$, this stepsize condition is much relaxed than that of the Chambolle and Pock's PDA \eqref{pda_CP}, which is $\tau\sigma L^2<1$.
Therefore, GRPDA permits larger primal and dual stepsizes, which is essential for fast practical convergence.
The experimental results given in \cite{ChY2020Golden} confirmed the benefits of allowing larger stepsizes.

%The golden ratio PDA was explained as an ADMM-like method with an indefinite proximal term $\frac{1}{2}\|x-x_{n-1}\|_M^2$, $M=\frac{1}{\tau}I-\sigma K^\top K$ is permitted to be indefinite, because $\tau\sigma L^2 < \psi$ and $\psi>1$, plus an extra linear term.

\subsection{Motivations and contributions}
The main contribution of this work is to introduce a linesearch strategy into GRPDA \eqref{GRPDA}.
Our motivations have two aspects. First, in many applications, especially when the matrix $K$ is large and dense, e.g., CT image reconstruction \cite{ImageBook04,SJP12PMB}, the exact spectral norm of $K$ can be very expensive to compute or estimate. On the other hand, even if the spectral norm of $K$ can be obtained, the stepsizes governed by the condition $\tau\sigma \|K\|^2<\psi$ is usually too conservative for fast practical convergence. Hence, our goal in this paper is to incorporate linesearch into GRPDA which can significantly
accelerate the convergence while still theoretically guarantees convergence with desired convergence rate.

%These are also the major motivations of \cite{Malitsky2018A} for introducing linesearch strategy into the PDA.
%

In general, linesearch requires extra evaluations of  proximal operators and/or matrix-vector multiplications in every linesearch iteration. Interestingly, for many special yet important applications as pointed out in \cite{Malitsky2018A}, the proximal operator of $f^*$ is extremely simple, and as a result the linesearch procedure does not require any additional matrix-vector multiplications.
So, motivated by \cite{Malitsky2018A}, which introduced linesearch into the PDA \eqref{pda_CP}, we propose in this paper to incorporate linesearch into the GRPDA \eqref{GRPDA}. However, as will be seen in later
sections, our theoretical analysis on the stepsize behaviors generated by the linesearch is fundamentally
different from those presented in \cite{Malitsky2018A} or given in other literature.
In particular, our algorithm combining with linesearch not only does not assume \emph{a priori} knowledge
about the spectral norm of $K$, but also generates adaptive and potentially much larger stepsizes.
We establish global convergence as well as ergodic ${\cal O}(1/N)$ sublinear convergence rate in the general convex case, where $N$ denotes the iteration counter. When either one of the component functions is strongly convex, GRPDA with linesearch can be shown to converge at the faster ${\cal O}(1/N^2)$ ergodic sublinear rate. Furthermore, if both component functions are strongly convex, iterative and nonergodic linear convergence results are established.
Hence, even with the stepsize relaxations by the proposed linesearch, the global convergence as well as
theoretical convergence rates are still guaranteed to remain consistent with their counterparts without using linesearch. In addition, our numerical experiments show much practical benefits can be obtained
from the proposed linesearch strategies.
%
%We also present numerical results to illustrate the benefits achieved from the proposed linesearch strategy.

\subsection{Organization}
The rest of this paper is organized as follows. In Section \ref{sec:asmp}, we make our assumptions, provide some useful facts and further define some notations. Section \ref{sec-GRPDA-L} is devoted to the GRPDA with linesearch in the general convex case, while the cases when either one or both of the component functions are strongly convex are discussed in Section \ref{sec-acc-L}.
Our numerical results on minimax matrix game and LASSO problems are reported in Section \ref{sec-experiments} to show the benefits obtained by adopting the linesearch strategy. Finally, some concluding remarks are drawn in Section \ref{sec-conclusion}.

\section{Assumptions and preliminaries}
\label{sec:asmp}

\subsection{Assumptions and further notation}
Throughout the paper, we make the following blanket assumptions.
\begin{asmp}\label{asmp-1}
Assume that the set of solutions of \eqref{primal} is nonempty and,
in addition, there exists $\tilde x\in\text{ri}(\text{dom}(g))$ such that $K\tilde x\in\text{ri}(\text{dom}(f))$.
\end{asmp}

Under Assumption \ref{asmp-1}, it follows from \cite[Corollaries 28.2.2 and 28.3.1]{Rockafellar1970Convex} that
${\bar x}\in \R^q$ is a solution of \eqref{primal} if and only if there exists ${\bar y}\in \R^p$ such that $({\bar x},{\bar y})$ is a saddle point of $\cL(x,y)$, i.e.,
$\cL({\bar x},y)\leq\cL({\bar x},{\bar y})\leq\cL(x,{\bar y})$ for all $(x,y)\in \R^q\times \R^p$. As such, $({\bar x},{\bar y})$ is a solution of the minimax problem \eqref{saddle_point} and $\bar y$ is a solution of the dual problem \eqref{dual}.
We denote the set of solutions of \eqref{saddle_point} by $\Omega$, which is nonempty under  Assumption \ref{asmp-1} and
characterized by
\[\nonumber %\label{def:cS}
\Omega := \{({\bar x},{\bar y})\in \R^q\times \R^p: \; 0 \in \partial g({\bar x}) + K^\top {\bar y} \text{~~and~~} 0\in \partial f^*({\bar y}) - K{\bar x}\}.
\]
Hereafter, we let $({\bar x}, {\bar y}) \in \Omega$ be a generic saddle point. When $g$ (resp. $f^*$) is strongly convex, the primal optimal solution $\bar x$ (resp. the dual optimal solution $\bar y$) is unique.
Define
\ben
P(x) := g(x)-g(\bar{x}) + \langle K^\top \bar{y}, x-\bar{x}\rangle,~\forall x\in \R^q, \\
D(y) := f^*(y)-f^*(\bar{y}) - \langle K \bar{x}, y-\bar{y}\rangle,~\forall y\in \R^p.
\een
By subgradient inequality, it is clear that $P(x)\geq 0$ and $D(y)\geq 0$ for all $x\in\R^q$ and $y\in\R^p$.
Apparently, $P(x)$ and $D(y)$ are convex in $x$ and $y$, respectively.
The primal-dual gap function is defined by $G(x,y) := \mathcal{L}(x,\bar{y})-\mathcal{L}(\bar{x},y)$ for $(x,y)\in\R^q\times\R^p$. It is easy to verify that
\be\label{G}
G(x,y) = P(x)+D(y)\geq0,~\forall (x,y)\in \R^q\times\R^p.
\ee
Note that the functions $P(\cdot)$ and $D(\cdot)$ depend on the saddle point $(\bar x,\bar y)\in \Omega$.
Nonetheless, we omit this dependence in the notation $P(\cdot)$ and $D(\cdot)$ since it is always clear from  the context which saddle point is under consideration, and similarly for the primal-dual gap function $G(\cdot,\cdot)$. This measure of primal-dual gap function has been used in, e.g., \cite{Chambolle2011A,Chambolle2016ergodic,Malitsky2018A}, to establish convergence rate results for primal-dual type methods. We also adopt the measure in this paper.

\begin{asmp}\label{asmp-2}
  Assume that the proximal operators of the component functions $f$ and $g$ either have closed form formulas or can be evaluated efficiently.
\end{asmp}

In applications such as signal and image processing and machine learning, the component functions $f$ and $g$ usually enforce data fitting and regularization and hence, often preserve simple structures so that their proximal operators can be computed efficiently or just have closed form formulas. Examples of such functions are abundant, see, e.g., \cite[Chapter 6]{Beck2017book}. Therefore, Assumption \ref{asmp-2} is fulfilled in diverse applications. Note that  the proximal operators of $f^*$ and $g^*$ are also easily computable under Assumption \ref{asmp-2} due to the Moreau decomposition theorem \cite[Theorem 31.5]{Rockafellar1970Convex}.

\subsection{Facts and identities}
The following simple facts and identities will be used repeatedly in the convergence analysis.
\begin{fact}\label{fact_proj}
Let $h: \R^m\rightarrow (-\infty, +\infty]$ be an extended real-valued closed proper and $\gamma$-strongly convex function with modulus $\gamma \geq 0$. Then for any $\tau>0$ and $x\in \R^m$, it holds that $z = \prox_{\tau h}(x)$ if and only if
$h(y) \geq h(z)+ {1\over \tau}\langle x-z, y-z\rangle + {\gamma \over 2}\|y-z\|^2$ for all $y\in \R^m$.
\end{fact}
\begin{fact}\label{fact_ab}
Let $\{a_n: n\geq 1\}$ and $\{b_n: n\geq 1\}$ be real and nonnegative sequences. If $a_{n+1} \leq a_n-b_n$ for all $ n\geq1$, then $\lim_{n\rightarrow\infty}a_n$ exists and $\lim_{n\rightarrow \infty} b_n = 0$.
\end{fact}
For any $x, y, z\in \bR^m$ and $\alpha\in\bR$, there hold
\be
2\langle x-y, x-z\rangle&=&  \|x-y\|^2 +  \|x-z\|^2  -\|y-z\|^2,\label{id}\\
\|\alpha x+(1-\alpha)y\|^2&=& \alpha \|x\|^2+(1-\alpha)\|y\|^2-\alpha(1-\alpha)\|x-y\|^2.\label{id2}
\ee
Verifications of these identities are elementary.

\section{GRPDA with linesearch --- Convex case}
\label{sec-GRPDA-L}

%In the rest of this paper, we let $\psi\in (1,\phi)$ and $\varphi :=\frac{1+\psi}{\psi^2} \in (1,2)$, where $\phi = {\sqrt{5}+1\over 2}$ is the golden ratio.

As pointed out in Section \ref{sec:intro}, in some applications it can be very expensive to estimate the spectral norm  of $K$. Furthermore, the stepsizes governed by the condition $\tau\sigma \|K\|^2 < \psi \in (1,\phi]$ are usually too conservative in the applications. To address these issues, we introduce a linesearch strategy into GRPDA to choose stepsizes adaptively.
Within each linesearch step, only the dual variable needs to be updated.
The resulting algorithm, called GRPDA-L, is stated in Algorithm \ref{algo2}.
%Recall that $\phi = {\sqrt{5}+1\over 2}$ denotes the golden ratio. Moreover,

%
%\vskip5mm
%\hrule\vskip2mm
%\begin{algo}
%[Linesearch]\label{algo-linesearch}
%{~}\vskip 1pt {\rm
%\begin{description}
%\textbf{Input:} $\tau_{n-1}$, $\tau$ and $\beta\in(0,+\infty)$, $\mu\in(0,1), \zeta(0,+\infty)$, and points $x_n\in\bR^q$, $y_{n-1}\in\bR^p$.\\
%(i).~ Compute
%\ben
% y_{n}&=&\prox_{\beta\tau f^*}(y_{n-1}+\beta\tau Kx_{n}).
% \een
%(ii).~ Break linesearch if
%\be\label{line-ineq}
%\beta \tau\|K^\top y_{n}-K^\top y_{n-1}\|^2\leq\frac{\zeta}{\tau_{n-1}}\|y_{n}-y_{n-1}\|^2.
%\ee
%\textbf{Output:} $\tau_n=\tau$ and $y_n$. Otherwise, set $\tau:= \mu\tau$ and go to (i).
%  \end{description}
%}
%\end{algo}
%\vskip1mm\hrule\vskip5mm

% \newpage

\vskip5mm
\hrule\vskip2mm
\begin{algo}
[GRPDA-L]\label{algo2}
{~}\vskip 1pt {\rm
\begin{description}
\item[{\em Step 0.}] Choose $x_0 = z_{0} \in \bR^q,$ $y_0\in \bR^p$, $\psi\in(1,\phi)$, $\sigma\in(0,1)$, $\beta>0$, $\mu\in(0,1)$ and $\tau_0 > 0$.
    Set $\varphi=\frac{1+\psi}{\psi^2} \in (1,2)$ and $n=1$.
\item[{\em Step 1.}]Compute
\be
  z_{n}&=&\frac{\psi-1}{\psi} x_{n-1} + \frac{1}{\psi}z_{n-1}, \label{z_updating} \\
 x_{n}&=&\prox_{\tau_{n-1} g}(z_{n}-\tau_{n-1} K^\top y_{n-1}).\label{x_updating}
 \ee
\item[{\em Step 2.}] Let $\tau= \varphi\tau_{n-1}$ and compute
\be\label{y1_updating}
   y_{n}=\prox_{\beta\tau_{n} f^*}(y_{n-1}+\beta\tau_{n} Kx_{n}),
\ee
where $\tau_n = \tau \mu^i$ and $i$ is the smallest nonnegative integer such that
\be\label{y-ls}
\sqrt{\beta\tau_{n}}\|K^\top y_{n}-K^\top y_{n-1}\|\leq\sigma\sqrt{\psi / \tau_{n-1}}\|y_{n}-y_{n-1}\|.
\ee
\item[{\em Step 3.}] Set $n\leftarrow n + 1$ and return to Step 1.
  \end{description}
}
\end{algo}
\vskip1mm\hrule\vskip5mm

%\begin{rem}
%In Algorithm \ref{algo2}, the parameter $\overline{\tau}>0$ is given to ensure that $\{\tau_n: n\geq 1\}$ is upper bounded, which guarantees the finite termination of the linesearch procedure. We emphasize that although the spectral norm $L = \|K\|$ appears in the condition $\overline{\tau} \geq \tau_0 \geq  \frac{\sigma^2\psi\mu}{\overline{\tau}\beta L^2}$, this does not mean that it has to be computed in advance. To satisfy the parameter condition, one can, for example, first choose $u\in\R^q$ and $v\in\R^p$ arbitrarily such that $\|u\| = \|v\|=1$ and $\alpha := |\langle Ku, v\rangle| \neq 0$, then set $\overline{\tau} \geq \sqrt{\frac{\sigma^2\psi\mu}{\beta \alpha^2}}$, and finally choose $\tau_0\in \big[ \frac{\sigma^2\psi\mu}{\overline{\tau}\beta \alpha^2}, \overline{\tau}\big]$.
%In summary, knowledge about the spectral norm parameter $L$ is not required.
%\end{rem}

From Step 2 of Algorithm \ref{algo2}, the linesearch procedure may require to compute $\prox_{\beta\tau_n f^*}$ and $K^\top y_{n}$ repeatedly to find a proper $\tau_n$ at each iteration. However,  as pointed out in \cite[Remark 2]{Malitsky2019Golden}, the linesearch procedure becomes extremely simple when the proximal operator of $\lambda f^*$, where $\lambda >0$, is linear or affine. Some examples are listed below.
\begin{enumerate}
\item[(a)] $\prox_{\lambda f^*}(u)=u-\lambda c$ when $f^*(y) = \langle c,y\rangle$ for some $c\in\R^p$,

\item[(b)] $\prox_{\lambda f^*}(u)=\frac{1}{1+\lambda}(u+\lambda b)$ when $f^*(y)= \frac 1 2 \|y-b\|^2$ for some $b\in\R^p$,

\item[(c)] $\prox_{\lambda f^*}(u)=u+\frac{b-\langle u,a\rangle}{\|a\|^2}a$ when $f^*(y)$ is the indicator function of the hyperplane $H = \{u: \langle a,u\rangle= b\}$ for some $a\in\R^p$ and $b\in\R$.
\end{enumerate}
In all these cases, the evaluation of $\prox_{\lambda f^*}$ is very simple, and it is unnecessary to compute $K^\top y_n$ repeatedly since it can be obtained by combining some already computed quantities.
 Therefore, in these cases the linesearch step is quite cheap and does not require any additional matrix-vector multiplications.
Furthermore, if necessary, one can always exchange the roles of the primal and the dual variables in problem (\ref{saddle_point}) to take advantage of the above mentioned structure.
Also note that at each iteration the increasing ratio of stepsizes $\{\tau_n \}$ is upper bounded by $\varphi\in(1, 2)$ as $\tau_n\leq \varphi\tau_{n-1}$, and as suggested in \cite{Malitsky2019Golden}, one
default choice could let $\psi=1.5$ so that $\varphi=(1+\psi)/\psi^2 = 10/9$.
% Note that in \eqref{y1_updating} the dual stepsize is set to be $\beta \tau_n$, where $\tau_n$ satisfies \eqref{y-ls}.
Here, the parameter $\beta>0$ in Algorithm \ref{algo2}, also appeared in Algorithm \ref{algo-acc3-L},
is introduced to scale the primal and the dual variables so that they will converge in a weighted
balance way. Similar settings has also been done in  \cite{ChY2020Golden} for GRPDA without linesearch.

The following lemma shows that the linesearch step of Algorithm \ref{algo2} is well-defined.
In addition, it establishes some important properties on sequences
$\{\tau_n: n\geq 1 \}$ and $\{\delta_n: n\geq 1\}$,
with $\delta_n := \tau_n/\tau_{n-1}$, which are essential for establishing the convergence results of Algorithm \ref{algo2}.
Since the proof of Lemma \ref{lem_bound} is rather technical and complicated, for fluency of
the overall paper, we put the proof in Appendix \ref{proof:lem-bound}.

\begin{lem}\label{lem_bound}
Let $\underline{\tau}:=  \frac{\sigma\sqrt{\psi}}{L\sqrt{\beta\varphi}}  > 0$. Then, we have the following properties.
(i) The linesearch step of Algorithm \ref{algo2}, i.e., Step 2, always terminates.
(ii) For  any $\rho \in (0, 1)$,
there exists an infinite subsequence $\{n_k: k \ge 1\} \subseteq \{1,2, \ldots\}$
such that  $\tau_{n_k} \geq \underline{\tau}$ and $\delta_{n_k} \ge \rho$.
(iii) For any integer $N >0$, we have $|\mathcal{K}_N| \ge \hat{c} N$ for some
constant $\hat{c} > 0$, where $\mathcal{K}_N = \{1 \le n \le N: \tau_n \ge \underline{\tau}\}$
and $|\mathcal{K}_N|$ is the cardinality of the set $\mathcal{K}_N$, which implies
$\sum_{n=1}^N \tau_n \ge \underline{c} N$
with $\underline{c} = \hat{c} \underline{\tau} $.
\end{lem}

We emphasize that the linesearch procedure adopted by Algorithm \ref{algo2} is motivated but
theoretically very different from that of
\cite[Algorithim 1]{Malitsky2018A}. In fact, the sequence $\{\tau_n: n\geq 1\}$ generated by \cite[Algorithim 1]{Malitsky2018A}
is uniformly bounded below by some positive constant, see \cite[Lemma 3.3 (ii)]{Malitsky2018A}.
In contrast, as seen in (iii) of Lemma~\ref{lem_bound},
only a subsequence $\{\tau_{n_k}: k\geq 1\}$, a fraction of $\{\tau_n: n\geq 1\}$, is guaranteed to
have uniform lower bound $\underline{\tau} > 0$.
 Similar arguments also apply to Algorithms \ref{algo-acc-L} and \ref{algo-acc3-L} in Section \ref{sec:g-str}.

\subsection{Useful lemmas}
We next present two useful lemmas, which play critical roles in the convergence analysis.
Hereafter, we always fix $(\bar{x}, \bar{y})\in \Omega$ arbitrarily without further mentioning.

\begin{lem}\label{lem1}
Let $\{(z_n,x_n,y_n): n\geq 1\}$ be generated by Algorithm \ref{algo2}. Then, for any
$(\bar{x}, \bar{y})\in \Omega$, there holds
\be
\tau_n G(x_n,y_n)
&\leq&\langle x_{n+1}-z_{n+1}, \bar{x}-x_{n+1}\rangle+ \frac{1}{\beta} \langle  y_{n}-y_{n-1}, \bar{y}-y_{n}\rangle+ \psi\delta_{n} \langle x_{n}-z_{n+1}, x_{n+1}- x_n \rangle \nonumber\\
&&+ \tau_n\big\langle K^\top (y_{n}- y_{n-1}), x_n-x_{n+1}\big\rangle. \label{lem1:ineq}
\ee
\end{lem}
\begin{proof}
It follows from  (\ref{x_updating}), (\ref{y1_updating}) and Fact  \ref{fact_proj} that
\begin{align}
  \big\langle x_{n+1}-z_{n+1} + \tau_{n}K^\top y_{n}, ~~ \bar{x}-x_{n+1}\big\rangle
  &\geq \tau_{n}\big(g(x_{n+1})-g(\bar{x})\big),\label{temp_01}\\
\langle x_{n}-z_{n} + \tau_{n-1}K^\top y_{n-1},~~ x_{n+1}-x_{n}\rangle
&\geq \tau_{n-1} \big(g(x_{n})-g(x_{n+1})\big), \label{tem1} \\
 \big\langle \frac{1}{\beta}(y_{n}-y_{n-1}) - \tau_{n}K x_{n}, ~~\bar{y}-y_{n} \big\rangle
&\geq \tau_{n}\big(f^*(y_{n})-f^*(\bar{y})\big).\label{temp_001}
\end{align}
Multiplying (\ref{tem1}) by $\delta_n = \tau_{n} /\tau_{n-1}$ and using the fact
$x_{n}-z_{n} = \psi (x_{n}-z_{n+1})$, which follows from \eqref{z_updating}, we obtain
\be\label{temp_02}
\left\langle \psi\delta_n(x_{n}-z_{n+1}) + \tau_{n}K^\top y_{n-1}, ~~x_{n+1}-x_{n}\right\rangle
\geq \tau_{n}\big(g(x_{n})-g(x_{n+1})\big).
\ee
Direct calculations show that a summation of (\ref{temp_01}), (\ref{temp_001}) and (\ref{temp_02}) gives
\begin{align}
%&&\langle x_{n+1}-z_{n+1}, \bar{x}-x_{n+1}\rangle+  \frac{1}{\beta} \langle y_{n}-y_{n-1}, \bar{y}-y_{n}\rangle+ \psi\delta_n \langle x_{n}-z_{n+1}, x_{n+1}- x_{n} \rangle \nonumber\\
%&& + \tau_{n} \big\langle  K^\top y_{n}, ~~ \bar{x}-x_{n+1}\big\rangle
% - \tau_{n} \big\langle  K x_{n}, ~~\bar{y}-y_{n} \big\rangle
%+ \tau_{n} \left\langle  K^\top y_{n-1}, ~~x_{n+1}-x_{n}\right\rangle \\
%
%
\langle x_{n+1}-z_{n+1}, \; & \bar{x}-x_{n+1}\rangle +  \frac{1}{\beta} \langle y_{n}-y_{n-1}, \bar{y}-y_{n}\rangle+ \psi\delta_n \langle x_{n}-z_{n+1}, x_{n+1}- x_{n} \rangle \nonumber\\
+ &  \tau_n\big\langle K^\top (y_{n}- y_{n-1}),~~ x_{n}-x_{n+1}\big\rangle-\tau_n \big\langle K^\top \bar{y}, ~~ x_{n}-\bar{x}\big\rangle + \tau_{n} \langle K\bar{x}, ~~ y_{n}-\bar{y}\rangle\nonumber\\
\geq &\tau_{n}\big(f^*(y_{n})-f^*(\bar{y})\big)+\tau_{n}\big(g(x_{n})-g(\bar{x})\big), \nonumber %\label{temp_04}
\end{align}
which, by the definition of $G(\cdot,\cdot)$ in (\ref{G}),  implies \eqref{lem1:ineq} immediately.
\end{proof}

\begin{lem}\label{lem2}
Let $\{(z_n,x_n,y_n): n\geq 1\}$ be generated by Algorithm \ref{algo2}. For $n\geq 1$, define
\be
 a_n &:=&\frac{\psi}{\psi-1}\|z_{n+1}-\bar{x}\|^2+\frac{1}{\beta}\|y_{n-1}-\bar{y}\|^2,\label{a_n}\\
 b_n &:=&\psi\delta_n\|z_{n+1}-x_n\|^2+
 (1-\sigma)\big(\psi\delta_n\|x_{n+1}-x_n\|^2+\frac{1}{\beta}\|y_{n}-y_{n-1}\|^2\big). \label{b_n}
\ee
Then, it holds that  $a_{n+1}\leq a_{n}-b_n$ for $n\geq 1$.
\end{lem}
\begin{proof}
Fix $n\geq 1$.
First, it is easy to verify from
 $\varphi= (1+\psi)/\psi^2$ and $\delta_n= \tau_n/\tau_{n-1}\leq\varphi$  that
\be\label{lem3.3-c1}
1+\frac{1}{\psi}-\psi\delta_n\geq1+\frac{1}{\psi}-\psi\varphi=0.
\ee
It follows from \eqref{y-ls} and Cauchy-Schwarz inequality that
\be
2\tau_n\|K^\top y_{n}-K^\top y_{n-1}\|\|x_{n+1}-x_n\|
%&\leq& 2\sigma\frac{\sqrt{\psi\delta_n}}{\sqrt{\beta}} \|y_{n}-y_{n-1}\|\|x_{n+1}-x_n\|\nonumber\\
&\leq&\sigma\big(\psi\delta_n\|x_{n+1}-x_n\|^2+ \frac{1}{\beta}\|y_{n}-y_{n-1}\|^2\big). \label{lem3.3-c2}
\ee
Furthermore, by Lemma \ref{lem1}, identity (\ref{id}) and Cauchy-Schwarz inequality, we have
\be\label{ineq11}
&&\|x_{n+1}-\bar{x}\|^2+\frac{1}{\beta}\|y_{n}-\bar{y}\|^2+ 2 \tau_n G(x_n,y_n)\nonumber\\
&\leq&\|z_{n+1}-\bar{x}\|^2+\frac{1}{\beta}\|y_{n-1}-\bar{y}\|^2 + 2\tau_n\|K^\top y_{n}-K^\top y_{n-1}\|\|x_{n+1}-x_n\|  \\
&&-\psi\delta_n\|z_{n+1}-x_n\|^2-(1-\psi\delta_n)\|x_{n+1}-z_{n+1}\|^2
-\psi\delta_n\|x_{n+1}-x_n\|^2-\frac{1}{\beta}\|y_{n}-y_{n-1}\|^2. \nonumber
\ee
Since $x_{n+1}=\frac{\psi}{\psi-1}z_{n+2}-\frac{1}{\psi-1}z_{n+1}$, which follows from
(\ref{id2}) and \eqref{z_updating}, we deduce
\be\label{x-to-z}
\|x_{n+1}-\bar{x}\|^2
&=& {\psi\over \psi-1} \|z_{n+2}-\bar{x}\|^2- {1 \over \psi-1} \|z_{n+1}-\bar{x}\|^2 + {\psi \over (\psi-1)^2} \|z_{n+2}-z_{n+1}\|^2\nonumber\\
&=& {\psi\over \psi-1} \|z_{n+2}-\bar{x}\|^2- {1 \over \psi-1} \|z_{n+1}-\bar{x}\|^2+\frac{1}{\psi}\|x_{n+1}-z_{n+1}\|^2,
\ee
where the second equality is due to $z_{n+2}-z_{n+1} = {\psi-1\over\psi} (x_{n+1} - z_{n+1})$.
Combining \eqref{x-to-z} with (\ref{ineq11}), we obtain
\be\label{ineq_rate1}
&&\frac{\psi}{\psi-1}\|z_{n+2}-\bar{x}\|^2+\frac{1}{\beta}\|y_{n}-\bar{y}\|^2+ 2 \tau_n G(x_n,y_n)\nonumber\\
&\leq&\frac{\psi}{\psi-1}\|z_{n+1}-\bar{x}\|^2+\frac{1}{\beta}\|y_{n-1}-\bar{y}\|^2 + 2\tau_n\|K^\top y_{n}-K^\top y_{n-1}\|\|x_{n+1}-x_n\| \nonumber \\
&&-\psi\delta_n\|z_{n+1}-x_n\|^2-(1+\frac{1}{\psi}-\psi\delta_n)\|x_{n+1}-z_{n+1}\|^2
-\psi\delta_n\|x_{n+1}-x_n\|^2-\frac{1}{\beta}\|y_{n}-y_{n-1}\|^2. \nonumber \\
%
% \frac{\psi}{\psi-1}\|z_{n+2}-\bar{x}\|^2+\frac{1}{\beta}\|y_{n}-\bar{y}\|^2+ 2 \tau_n G(x_n,y_n)
&\leq&\frac{\psi}{\psi-1}\|z_{n+1}-\bar{x}\|^2+\frac{1}{\beta}\|y_{n-1}-\bar{y}\|^2 -\psi\delta_n\|z_{n+1}-x_n\|^2 \nonumber \\
& & - (1 - \sigma)\psi\delta_n\|x_{n+1}-x_n\|^2- \frac{ 1-\sigma }{\beta}\|y_{n}-y_{n-1}\|^2, \label{lem3.3-i2}
\ee
where the second inequality follows from \eqref{lem3.3-c1} and \eqref{lem3.3-c2}.
Finally, by the definitions of $a_n$ and $b_n$ in \eqref{a_n} and \eqref{b_n}, respectively, and the fact that $G(x_n,y_n) \geq 0$, \eqref{lem3.3-i2} implies $a_{n+1}\leq a_{n}-b_n$ immediately. %This completes the proof.
\end{proof}

\subsection{Convergence results}
Now, we are ready to establish global convergence and ergodic sublinear convergence rate of Algorithm \ref{algo2}.
\begin{thm}[Global convergence]\label{thm12}
Let $\{(z_n,x_n,y_n): n\geq 1\}$ be generated by Algorithm \ref{algo2}.
 Then  $\{(x_{n}, y_n): n\geq 1\}$ converges to a solution of (\ref{saddle_point}).
\end{thm}
\begin{proof}
Let $\rho\in(0,1)$ and $\underline{\tau} > 0$ be defined in Lemma \ref{lem_bound}.
By (ii) of Lemma \ref{lem_bound}, there exists an infinite sequence $\{n_k: k\ge 1\}$ such that
 $\tau_{n_k} \geq \underline{\tau}$ and $\delta_{n_k} \ge \rho$.
By Lemma \ref{lem2}, we have
\begin{eqnarray}\label{lim-1}
 %&&
\sum_{k=1}^\infty \delta_{n_k} \left(\|z_{n_k+1} - x_{n_k}\|^2 + (1-\sigma)\|x_{n_k+1} - x_{n_k}\|^2 \right)  \leq \sum_{n=1}^\infty b_n < \infty.
% \nonumber \\
% &\le & \sum_{n=1}^\infty \delta_n \left(\|z_{n+1} - x_n\|^2 + (1-\sigma)\|x_{n+1} - x_n\|^2 \right) < \infty.
\end{eqnarray}
Since $\lim_{k \to \infty} \sum_{k=1}^{\infty} \delta_{n_k} = \infty$,
which together with \eqref{lim-1} implies that
 there exists a subsequence of $\{n_k: k \geq 1\}$, still denoted
 as $\{n_k: k \geq 1\}$, such that
\begin{equation}\label{lim-2}
\lim\limits_{k \rightarrow\infty} \|z_{n_k+1} - x_{n_k}\| = 0
\quad \mbox{ and } \quad \lim\limits_{n\rightarrow\infty}\|x_{n_k+1}-x_{n_k}\|=0.
\end{equation}
In addition, it follows from Lemma \ref{lem2} and Fact \ref{fact_ab} that $\lim\limits_{n\rightarrow\infty}\|y_{n+1}-y_n\|=0$, $\lim\limits_{n\rightarrow\infty}\|z_n-{\bar x}\|$ and $\lim\limits_{n\rightarrow\infty}\|y_n-{\bar y}\|$ exist.
Thus,   $\{z_{n}: n\geq 1\}$ and $\{y_{n}: n\geq 1\}$ are bounded and, by \eqref{z_updating}, so is $\{x_{n}: n\geq 1\}$.
Therefore, there exist a subsequence of $\{n_k: k\geq 1\}$, still denoted as
$\{n_k: k\geq 1\}$, and $(x^*,y^*)$ such that
$\lim\limits_{k\rightarrow\infty} x_{n_k} = x^*$ and $\lim\limits_{k\rightarrow\infty} y_{n_k} = y^*$,
which together with \eqref{lim-2} also implies that $\lim_{k \to \infty} x_{n_k + 1} = \lim_{k \to \infty} z_{n_k+1} = x^*$.
Now, similar to \eqref{temp_01} and \eqref{temp_001}, for any $(x,y)$, there hold
\begin{equation} \label{sub-optmal}
\left\{\ba{l}
\langle x_{n_k+1}-z_{n_k+1} + \tau_{n_k} K^\top y_{n_k}, ~ x-x_{n_k+1}\rangle\geq \tau_{n_k} \big(g(x_{n_k+1})-g(x)\big), \smallskip \\
\big\langle {1\over\beta} (y_{n_k}-y_{n_k-1}) - \tau_{n_k} K x_{n_k}, ~y-y_{n_k}\big\rangle\geq \tau_{n_k} \big(f^*(y_{n_k})-f^*(y)\big).
\ea
\right.
\end{equation}
Then, dividing $\tau_{n_k} \geq \underline{\tau} > 0$ from both sides of \eqref{sub-optmal},
taking into account  that both $g$ and $f^*$ are closed (and thus lower semicontinuous) and
letting $k\rightarrow\infty$, we obtain
\begin{equation} \label{thm12-1}
\langle  K^\top y^*, ~ x-x^*\rangle\geq  g(x^*)-g(x) \text{~~and~~} - \langle    K x^*, ~y-y^*\rangle\geq  f^*(y^*)-f^*(y).
\end{equation}
Since \eqref{thm12-1} holds for any $(x,y)\in \R^q\times \R^p$, we have $(x^*,y^*)\in \Omega$.
Note that Lemma \ref{lem2} holds for any $(\bar{x},\bar{y})\in \Omega$.
Therefore, $(\bar{x},\bar{y})$ can be replaced by $(x^*,y^*)$
in the definition of $\{a_n: n\geq 1\}$ in \eqref{a_n}.
As such, we have $\lim_{k\rightarrow\infty} a_{n_k} = 0$ since
$\lim_{k\rightarrow\infty} z_{n_k+1} = x^*$ and
$\lim_{k\rightarrow\infty} y_{n_k-1} = \lim_{k\rightarrow\infty} y_{n_k} = y^*$.
Since $\{a_n:  n\geq1\}$ is monotonically nonincreasing, it follows that $\lim_{n\rightarrow\infty} a_n = 0$. Therefore,
$\lim_{n\rightarrow\infty}(z_{n},y_{n}) = (x^*,y^*)$.
Again by \eqref{z_updating}, we have $\lim_{n\rightarrow\infty}x_n = x^*$.
This completes the proof.
\end{proof}

We now establish an $\cO(1/N)$  ergodic sublinear convergence rate of Algorithm \ref{algo2}.

\begin{thm}[Sublinear convergence]\label{thm22}
Let $\{(z_n,x_n,y_n,\tau_n): n\geq 1\}$ be the sequence generated by Algorithm \ref{algo2}
 and $\underline{c} >0$ be the constant given in Lemma~\ref{lem_bound}.
For any $N\geq 1$, define
\be\label{definition_XY}
X_N=\frac{1}{S_N}\sum_{n=1}^N\tau_n x_n~~\mbox{and}~~Y_N=\frac{1}{S_N} \sum_{n=1}^N\tau_n y_n~~\mbox{with}~~S_N=\sum_{n=1}^N\tau_n.
\ee
Then, it holds that
\begin{equation}\label{gap-sublinear}
G(X_N,Y_N)  \leq \frac{1}{2 \underline{c} N}\Big(\frac{\psi}{\psi-1}\|z_{2}-\bar{x}\|^2+\frac{1}{\beta}\|y_{0}-\bar{y}\|^2\Big).
\end{equation}
\end{thm}
\begin{proof}
First, it follows from \eqref{lem3.3-i2} and the definition of $b_n$ in \eqref{b_n} that
\ben
2 \tau_n  G(x_n,y_n)
 \leq \frac{\psi}{\psi-1} \big( \|z_{n+1}-\bar{x}\|^2 - \|z_{n+2}-\bar{x}\|^2\big)
 +\frac{1}{\beta} \big(\|y_{n-1}-\bar{y}\|^2-  \|y_{n}-\bar{y}\|^2 \big).
\een
By taking summation over $n = 1, \ldots, N$, we obtain
\be
 2\sum_{n=1}^N  \tau_n  G(x_n,y_n) \leq
\frac{\psi}{\psi-1}\|z_{2}-\bar{x}\|^2+\frac{1}{\beta}\|y_{0}-\bar{y}\|^2. \label{thm2-1}
\ee
Since $G(x,y) = P(x) + D(y)$ and $P(\cdot)$ and $D(\cdot)$ are convex,
it follows that
\be
\sum_{n=1}^N \tau_n  G(x_n,y_n)=\sum_{n=1}^N \tau_n  P(x_n)+\sum_{n=1}^N \tau_n D(y_n) \geq S_N\big(P(X_N)+D(Y_N)\big), \label{thm2-2}
\ee
where $S_N$, $X_N$ and $Y_N$ are defined in (\ref{definition_XY}).
Combining \eqref{thm2-1} and \eqref{thm2-2}, we obtain
\ben
G(X_N,Y_N) = P(X_N)+D(Y_N)\leq \frac{1}{2S_N}\big(\frac{\psi}{\psi-1}\|z_{2}-\bar{x}\|^2+\frac{1}{\beta}\|y_{0}-\bar{y}\|^2\big).
\een
By property (iii) in Lemma~\ref{lem_bound}, we have $S_N = \sum_{n=1}^N\tau_n \geq \underline{c} N$.
Hence, \eqref{gap-sublinear} holds.
\end{proof}

%
%Taking into account that $\{\tau_n\}$ is separated from zero and $\tau_n>\mu\vartheta$, where $\vartheta=\frac{\zeta}{\overline{\tau}\beta L^2}$ due to $\beta_n\equiv\beta$ from Lemma \ref{lem_bound}, we obtain $S_N>\mu\vartheta N$ and then the convergence rate $\cO(1/N)$ of Algorithm \ref{algo2} for the ergodic sequence $\{(X_N, Y_N)\}$ .

\section{GRPDA with linesearch --- Strongly convex case}
\label{sec-acc-L}

In this section, we introduce linesearch strategy into GRPDA \eqref{GRPDA} when  either $g$ or $f^*$ or both are strongly convex.
When $g$ or $f^*$ is strongly convex, it was shown in \cite{Chambolle2011A} that one can adaptively choose the primal and the dual stepsizes, as well as the inertial parameter, so that the PDA \eqref{pda_CP} achieves a faster $\cO(1/N^2)$ convergence rate. Similar results have been obtained in \cite{Malitsky2018A} for the PDA with linesearch. In this section, we propose an adaptive linesearch strategy
for GRPDA \eqref{GRPDA}.

Apparently, the minimax problem \eqref{saddle_point} is equivalent to
\be\label{saddle_point2}
\max_{x\in \R^q}\min_{y\in \R^p} ~~f^*(y) + \langle -K^\top y, x \rangle  -g(x).
\ee
Then, by swapping ``$\max_{x\in \R^q}$" with ``$\min_{y\in \R^p}$" and $(g,K,x,q)$ with $(f^*,-K^T,y,p)$,
\eqref{saddle_point2} is reducible to \eqref{saddle_point}.
Therefore, we only present the analysis for the case when $g$ is strongly convex in Section \ref{sec:g-str}, while the case when $f^*$ is strongly convex can be treated similarly and is thus omitted.
Finally, we establish in Section \ref{sec-gf-strongly} nonergodic linear convergence results when both $g$ and $f^*$ are strongly convex.

\subsection{The case when $g$ is strongly convex}
\label{sec:g-str}
In this subsection, we assume that $g$ is $\gamma_g$-strongly convex, i.e., it holds  for some $\gamma_g>0$  that
\ben
g(y) \geq g(x) + \langle u, y-x\rangle +
\frac{\gamma_g}{2}\|y-x\|^2,~~\forall \, x, y\in \R^q, \; \forall\, u\in \partial g(x),
\een
where the strongly convex parameter $\gamma_g>0$ is known in our algorithm.
Then, the following algorithm that incorporates linesearch into GRPDA \eqref{GRPDA} and
exploits the strong convexity of $g$ achieves faster convergence.

\vskip5mm
\hrule\vskip2mm
\begin{algo}
[Accelerated GRPDA with linesearch when $g$ is $\gamma_g$-strongly convex]\label{algo-acc-L}
{~}\vskip 1pt {\rm
\begin{description}
\item[{\em Step 0.}] Let $\psi_0 = 1.3247...$ be the unique real root of $\psi^3-\psi-1 = 0$. Choose $\psi\in(\psi_0,\phi)$, $\beta_0>0$, $\tau_0>0$ and $\mu\in(0,1)$.
    Choose $x_0=z_{0}\in \bR^q$ and $y_0\in \bR^p$.
    Set $\varphi=\frac{1+\psi}{\psi^2}$ and $n=1$.
\item[{\em Step 1.}]Compute
\be
z_{n}&=& \frac{\psi-1}{\psi} x_{n-1} + \frac{1}{\psi}z_{n-1}, \label{z-gstrong} \\
x_{n}&=&\prox_{\tau_{n-1} g}(z_{n}-\tau_{n-1} K^\top y_{n-1}), \label{x-gstrong}\\
 \omega_n&=&\frac{\psi-\varphi}{\psi+\varphi\gamma_g\tau_{n-1}},\label{omega-gstrong}\\
 \beta_{n}&=&\beta_{n-1}(1+ \gamma_g\omega_n \tau_{n-1}). \label{beta-gstrong}
\ee
\item[{\em Step 2.}] Let $\tau= \varphi\tau_{n-1}$ and compute
\be\label{y-gstrong}
   y_{n}=\prox_{\beta_n\tau_{n} f^*}(y_{n-1}+\beta_n\tau_{n} Kx_{n}),
\ee
where $\tau_n = \tau \mu^i$ and $i$ is the smallest nonnegative integer such that
\be\label{y-ls-gstrong}
\sqrt{\beta_n\tau_{n}}\|K^\top y_{n}-K^\top y_{n-1}\|\leq \sqrt{\psi / \tau_{n-1}}\|y_{n}-y_{n-1}\|.
\ee
\item[{\em Step 3.}] Set $n\leftarrow n + 1$ and return to Step 1.
  \end{description}
}
\end{algo}
\vskip1mm\hrule\vskip5mm

\begin{rem}
As in \cite[Algorithm 4.1]{ChY2020Golden}, $\omega_n$ is used to update $\beta_n$, which plays a key role in establishing the $\cO(1/N^2)$ ergodic convergence rate of Algorithm \ref{algo-acc-L}. The condition $\psi > \psi_0$, where $\psi_0$ is the unique real root of $\psi^3-\psi-1=0$, ensures that $\psi> (1+\psi)/\psi^2 = \varphi$. Therefore, we have $\omega_n>0$ and
$\beta_{n} > \beta_{n-1} > 0$ for all $n\geq 1$.
\end{rem}

Next, we state a key lemma with respect to  Algorithm~\ref{algo-acc-L}, whose detailed proof is left to Appendix \ref{proof:lem-beta-cnn}
since it takes similar spirit with %, but needs certain modifications from,
that of Lemma~\ref{lem_bound}.

\begin{lem}\label{lem-beta-cnn}
  Let $\{(\tau_n,\beta_n): n\geq 0\}$ be the sequence generated by Algorithm \ref{algo-acc-L}. Then, we have the following properties.
(i) The linesearch step of Algorithm~\ref{algo-acc-L}, i.e., Step 2, always terminates.
(ii)  There exists constant $c>0$ such that $\beta_n \geq c n^2$ for all $n\geq 1$.
(iii)  There exists a constant $\tilde{c} >0$ such that for any $N \ge 1$, we have $\sum_{n=1}^N \tau_n \le \tilde{c} \sum_{n \in \mathcal{S}_N} \tau_n$, where $\mathcal{S}_N := \left\{1 \le n \le N : \sqrt{\beta_{n}} \tau_n \ge 1/L \right\}$.
\end{lem}

We next establish the promised $\cO(1/N^2)$ ergodic convergence rate.
Note that in this case $\bar{x}$ is unique and will be simply denoted as $x^*$.
\begin{thm}[Accelerated sublinear convergence]
  \label{thm:gstrong}
 Let $\{(z_n,x_n,y_n, \beta_n, \tau_n): n\geq 1\}$ be the sequence generated by Algorithm \ref{algo-acc-L}.
Then, the following holds: \\
(a)   there exist constants $C_1, C_2 > 0$ such that
$\|z_{n+1}-x^*\| \leq C_1/n$ and  $\|x_{n+1}-x^*\| \leq C_2/n$ for all $n\geq 1$; \\
(b)  the sequence $\{y_n: n \ge 1\}$ is bounded and there exists a subsequence of
$\{y_n: n \ge 1\}$ converging to $y^*$ such that $(x^*, y^*)$ is a solution of (\ref{saddle_point}); \\
(c)  there exist a constant $C_3 >0$ such that  $G(X_N,Y_N) \leq C_3/N^2$ for any integer $N \ge 1$, where
\be\label{x_n}
S_N=\sum_{n=1}^N \beta_{n}\tau_{n}, ~~
X_N=\frac{1}{S_N}\sum_{n=1}^N \beta_{n}\tau_{n}x_n \text{~~and~~}
Y_N=\frac{1}{S_N}\sum_{n=1}^N \beta_{n}\tau_{n}y_n.
\ee
\end{thm}
\begin{proof}
Since $g$ is strongly convex, it follows from \eqref{x-gstrong} and Fact \ref{fact_proj} that
%$z_{n}-\tau_{n-1} K^\top y_{n-1} - x_{n} \in \tau_{n-1} \partial  g(x_{n})$ by  that
\begin{eqnarray}\label{ineq-ls-subgrad}
  \langle x_{n}-z_{n} + \tau_{n-1} K^\top y_{n-1}, ~x-x_{n}\rangle \geq \tau_{n-1} \Big(g(x_{n})-g(x)+\frac{\gamma_g}{2}\|x_{n}-x\|^2\Big), \; \forall\, x.
\end{eqnarray}
By passing $n+1$ to $n$ and $\bar x$ to $x$ in \eqref{ineq-ls-subgrad}, we obtain
\be\label{gstrong-ieq2}
\langle x_{n+1}-z_{n+1} + \tau_{n}K^\top y_{n}, ~~ \bar{x}-x_{n+1}\rangle\geq \tau_{n}\Big(g(x_{n+1})-g(\bar{x}) +\frac{\gamma_g}{2} \|x_{n+1}-\bar{x}\|^2\Big).
\ee
Similarly, by passing $x_{n+1}$ to $x$ in \eqref{ineq-ls-subgrad} and multiplying both sides by
$\delta_n =  \tau_n / \tau_{n-1}$, we obtain
\begin{eqnarray}\label{thm4.1-ieq2}
  \langle \delta_n(x_{n}-z_{n}) + \tau_{n} K^\top y_{n-1}, ~ x_{n+1}-x_{n}\rangle &\geq& \tau_{n} \Big(g(x_{n})-g(x_{n+1})+\frac{\gamma_g}{2}\|x_{n+1}-x_{n}\|^2\Big).
\end{eqnarray}
Similar to \eqref{temp_001}, it follows from \eqref{y-gstrong} that
\begin{eqnarray}\label{thm4.1-ieq3}
\big\langle {1\over \beta_n} (y_{n}-y_{n-1}) - \tau_n K x_{n}, ~\bar{y}-y_{n}\big\rangle \geq \tau_n \big(f^*(y_{n})-f^*(\bar{y})\big).
\end{eqnarray}
From \eqref{z-gstrong}, it is easy to derive $x_n-z_n = \psi(x_n - z_{n+1})$.
Then, by adding \eqref{gstrong-ieq2}-\eqref{thm4.1-ieq3} and using similar arguments as in Lemma \ref{lem1}, we obtain
\be\label{stronger-c}
\tau_n G(x_n,y_n)
&\leq&\langle x_{n+1}-z_{n+1}, \bar{x}-x_{n+1}\rangle+
\frac{1}{\beta_{n}} \big\langle y_{n}-y_{n-1}, \bar{y}-y_{n}\big\rangle
+\psi\delta_n \big\langle x_{n}-z_{n+1}, x_{n+1}- x_n\big\rangle \nonumber\\
&&+ \tau_n \langle K^\top y_{n}- K^\top y_{n-1}, x_n-x_{n+1}\rangle
- \frac{\gamma_g\tau_n}{2} \|x_{n+1}-\bar{x}\|^2 - \frac{\gamma_g\tau_n}{2} \|x_{n+1}-x_n\|^2.
\ee
By removing $-{\gamma_g\tau_n\over 2}  \|x_{n+1}-x_n\|^2\leq0$,  using \eqref{id} and Cauchy-Schwarz inequality, we obtain from \eqref{stronger-c} that
\be
&&(1+\gamma_g\tau_n)\|x_{n+1}-\bar{x}\|^2+\frac{1}{\beta_{n}}\|y_{n}-\bar{y}\|^2+ 2 \tau_n G(x_n,y_n)\nonumber\\
&\leq&\|z_{n+1}-\bar{x}\|^2+\frac{1}{\beta_{n}}\|y_{n-1}-\bar{y}\|^2  -\psi\delta_n\|z_{n+1}-x_n\|^2+(\psi\delta_n-1)\|x_{n+1}-z_{n+1}\|^2 \nonumber\\
&& -\psi\delta_{n}\|x_{n+1}-x_n\|^2-\frac{1}{\beta_{n}}\|y_{n}-y_{n-1}\|^2  + 2\tau_n\|K^\top (y_{n}- y_{n-1})\|\|x_{n+1}-x_n\|.
\label{jy-0316-1}
\ee
Plugging in (\ref{x-to-z}) and using  Cauchy-Schwarz inequality, we deduce
\be\label{ineq12}
&&(1+\gamma_g\tau_n)\frac{\psi}{\psi-1}\|z_{n+2}-\bar{x}\|^2
+\frac{1}{\beta_{n}}  \|y_{n}-\bar{y}\|^2+ 2 \tau_n G(x_n,y_n)\nonumber\\
&\leq&\frac{\psi+\gamma_g\tau_n}{\psi-1}\|z_{n+1}-\bar{x}\|^2+\frac{1}{\beta_{n}}\|y_{n-1}-\bar{y}\|^2 +\Big(\psi\delta_n-1-\frac{1+\gamma_g\tau_n}{\psi}\Big)\|x_{n+1}-z_{n+1}\|^2\nonumber\\
&&-\psi\delta_{n}\|z_{n+1}-x_n\|^2
-\psi\delta_{n}\|x_{n+1}-x_n\|^2-\frac{1}{\beta_{n}}\|y_{n}-y_{n-1}\|^2   \nonumber\\
&&+ 2\tau_n \|K^\top y_{n}-K^\top y_{n-1}\|\|x_{n+1}-x_n\|.
\ee
Recall that $\delta_n = \tau_n/\tau_{n-1}$. It follows from \eqref{y-ls-gstrong} and Cauchy-Schwarz inequality that
\be\label{cw-ieq}
2\tau_n \|K^\top y_{n}-K^\top y_{n-1}\|\|x_{n+1}-x_n\|&\leq& \psi\delta_{n}\|x_{n+1}-x_n\|^2+\frac{1}{\beta_{n}}\|y_{n}-y_{n-1}\|^2.
\ee
Furthermore,  $\psi\delta_n-1-\frac{1+\gamma_g\tau_n}{\psi} \leq \psi \varphi -1- \frac{1}{\psi}
- \frac{\gamma_g \tau_n}{\psi}= - \frac{\gamma_g \tau_n}{\psi}$
since $\delta_n\leq\varphi$.
Therefore, (\ref{ineq12}) implies
\be\label{ineq13}
&&(1+\gamma_g\tau_n)\frac{\psi}{\psi-1}\|z_{n+2}-\bar{x}\|^2
+\frac{1}{\beta_{n}}  \|y_{n}-\bar{y}\|^2+ 2 \tau_n G(x_n,y_n)\nonumber\\
&\leq&\frac{\psi+\gamma_g\tau_n}{\psi-1}\|z_{n+1}-\bar{x}\|^2+\frac{1}{\beta_{n}}\|y_{n-1}-\bar{y}\|^2
- \frac{\gamma_g \tau_n}{\psi} \|x_{n+1}-z_{n+1}\|^2.
%-\psi\delta_{n}\|z_{n+1}-x_n\|^2.
\ee
Note that $(1+\gamma_g\tau_n)\frac{\psi}{\psi-1} = \frac{\psi(1+\gamma_g\tau_n)}{\psi+\gamma_g\tau_{n+1}} \frac{\psi+\gamma_g\tau_{n+1}}{\psi-1}$.
It follows from $\tau_{n+1}\leq\varphi\tau_{n}$ that
\be\label{relax-coe}
  \frac{\psi(1+\gamma_g\tau_n)}{\psi+\gamma_g\tau_{n+1}}
&\geq& \frac{\psi(1+\gamma_g\tau_n)}{\psi+\gamma_g\varphi\tau_{n}}
= 1+\frac{\psi-\varphi}{{\psi}+\gamma_g\varphi\tau_n}\gamma_g\tau_n
=  1+\omega_{n+1}\gamma_g\tau_n,
\ee
and thus
\be\label{key-ineq}
(1+\gamma_g\tau_n)\frac{\psi}{\psi-1}
\geq (1+\omega_{n+1}\gamma_g\tau_n)\frac{\psi+\gamma_g\tau_{n+1}}{\psi-1}
=\frac{\beta_{n+1}}{\beta_n}\frac{\psi+\gamma_g\tau_{n+1}}{\psi-1}.
\ee
Define $A_{n}:=\frac{\psi+\gamma_g\tau_n}{2(\psi-1)}\|z_{n+1}-\bar{x}\|^2+ \frac{1}{2\beta_{n}}\|y_{n-1}-\bar{y}\|^2$.
Combining (\ref{ineq13}) and \eqref{key-ineq}, we deduce
\be\label{key-ineq2}
\beta_{n+1}A_{n+1}+\beta_{n}\tau_{n} G(x_n,y_n)
\leq \beta_{n}A_{n} - \frac{\beta_n \gamma_g \tau_n}{2\psi} \|x_{n+1}-z_{n+1}\|^2.
\ee
By summing \eqref{key-ineq2} over $n = 1, \ldots, N$, we obtain
\be\label{beta-A}
\beta_{N+1}A_{N+1}+\sum_{n=1}^N\beta_{n}\tau_{n} G(x_n,y_n)
+ \sum_{n=1}^N \frac{\beta_n \gamma_g \tau_n}{2\psi} \|x_{n+1}-z_{n+1}\|^2
\leq \beta_{1}A_{1}.
\ee
Then, the convexity of $G(\cdot,\cdot)$, $G(x_n,y_n) \ge 0$,
 the definition of $A_n$, \eqref{x_n}  and (\ref{beta-A}) imply that
\be
G(X_N,Y_N)&\leq&  {1\over  S_N}\sum_{n=1}^N\beta_{n}\tau_{n} G(x_n,y_n) \leq { \beta_{1}A_{1} \over S_N},\label{G-XY}\\
\|z_{N+2}-\bar{x}\|^2&\leq& \frac{2(\psi-1)}{\psi +  \gamma_g \tau_{N+1}}
{\beta_{1}A_{1} \over \beta_{N+1}} \leq {2\beta_{1}A_{1} \over \beta_{N+1}}.
\label{z-N}
\ee
From Lemma \ref{lem-beta-cnn}, there exists $c>0$ such that $\beta_n \geq c n^2$ for all $n\geq 1$.
Taking $x^* = \bar{x}$, then
 (\ref{z-N}) implies $\|z_{N+1}-x^*\| \leq C_1/N$ with $C_1 := \sqrt{2\beta_1A_1/c} > 0$.
Thus, \eqref{z-gstrong} implies  $\|x_{N+1}-x^*\| \leq C_2/N$ for some $C_2 >0$.
Hence, property (a) holds and the convergence of $\{ z_n: n\geq 1 \}$ and $\{x_n: n\geq 1\}$ to $x^*$ follows immediately.

Let $\mathcal{S} = \{n \in \mathcal{Z}^+ : \sqrt{\beta_n} \tau_n \ge 1/L \}$.
Properties (ii) and (iii) in Lemma~\ref{lem-beta-cnn}  imply that $|\mathcal{S}| =\infty$.
We next show by contradiction that there exists a subsequence $\{n_k: k \ge 1\} \subseteq \mathcal{S}$
 such that
\begin{equation}\label{sub-x-z}
\lim\limits_{k \to \infty } \|x_{n_k+1}-z_{n_k+1}\|/\tau_{n_k} = 0.
\end{equation}
Let $N\geq 1$ be arbitrarily fixed and $\theta =  (\psi - \varphi)\gamma_g/\psi >0$. By Lemma \ref{lem-beta-cnn} (ii) and \eqref{beta-gstrong}, we have
\ben
c (N+1)^2 \le
\beta_{N+1} = \beta_N \left(
1+ \frac{(\psi-\varphi)\gamma_g\tau_N}{\psi+\varphi\gamma_g\tau_N} \right)
\le \beta_N \big( 1+ \theta \tau_N \big)
\le
 \beta_{1} \prod_{n=1}^N \left( 1+ \theta \tau_n \right).
\een
Then, by Lemma \ref{lem-beta-cnn} (iii) we deduce
\be\label{sum_SNtn}
2\ln(N+1) - \ln(\beta_1/c) \leq \sum_{n=1}^N\ln(1 + \theta \tau_n)
\leq \theta \sum_{n=1}^N\tau_n
\leq \theta \tilde{c} \sum_{n\in {\cal S}_N} \tau_n.
\ee
On the other hand, it follows from \eqref{beta-A} that
\ben
{2\psi \beta_1 A_1  \over \gamma_g}
\geq \sum_{n\in {\cal S}} \beta_n\tau_n^3 \left({\|x_{n+1}-z_{n+1}\|  \over \tau_n }\right)^2
\geq {1\over L^2}\sum_{n\in {\cal S}}  \tau_n   \left({\|x_{n+1}-z_{n+1}\|  \over \tau_n }\right)^2,
\een
from which the existence of a subsequence $\{n_k: k \ge 1\} \subseteq \mathcal{S}$ such
that \eqref{sub-x-z} holds is guaranteed since \eqref{sum_SNtn} indicates that
$\sum_{n\in {\cal S}} \tau_n = \lim_{N\rightarrow\infty } \sum_{n\in {\cal S}_N} \tau_n = \infty$.

Now, it follows from \eqref{beta-A} that $\{ \beta_n A_n : n \ge 1\}$ is bounded,
which implies that $\{ y_n: n \ge 1\}$ is bounded.
Let $\{n_k : k \ge 1 \} \subseteq \mathcal{S}$ be the subsequence
satisfying \eqref{sub-x-z}, which then has a subsequence, still denoted as $\{n_k: k \ge 1\}  \subseteq \mathcal{S}$,
such that $\lim_{k \to \infty} y_{n_k} = y^*$.
Similar to \eqref{sub-optmal}, for any $(x,y)\in \R^q\times \R^p$, we have
\begin{equation}\label{sub-opt-2}
\left\{\ba{l}
\langle x_{n_k+1}-z_{n_k+1} + \tau_{n_k} K^\top y_{n_k}, ~ x-x_{n_k+1}\rangle\geq \tau_{n_k} \big(g(x_{n_k+1})-g(x)\big), \smallskip \\
\big\langle {1\over\beta_{n_k}} (y_{n_k}-y_{n_k-1}) - \tau_{n_k} K x_{n_k}, ~y-y_{n_k}\big\rangle\geq \tau_{n_k} \big(f^*(y_{n_k})-f^*(y)\big).
\ea
\right.
\end{equation}
Since $n_k \in \mathcal{S}$, we have $\sqrt{\beta_{n_k}} \tau_{n_k} \ge 1/L$, which
together with $\beta_n \ge c n^2$ for some $c >0$ implies
$\lim_{k \to \infty} \beta_{n_k} \tau_{n_k} = \infty$.
Hence,
 $\lim_{k \to \infty} \|y_{n_k}-y_{n_k-1}\|/(\beta_{n_k} \tau_{n_k}) =0$ since   $\{ y_n: n \ge 1\}$ is bounded.
Then, dividing $\tau_{n_k}$ from both sides of the two inequalities in \eqref{sub-opt-2},
taking $k \to \infty$, it follows from $\lim_{n \to \infty} x_{n} = x^*$,
$\lim_{k \to \infty} y_{n_k} = y^*$, lower semicontinuous of $g$ and $f^*$, and \eqref{sub-x-z} that
\eqref{thm12-1} holds, which implies $(x^*, y^*)$ is a saddle point of \eqref{saddle_point}.
Hence, property (b) holds.

Finally, it follows from \eqref{beta-gstrong} and $\omega_{n+1} < 1$ that
\ben
\beta_n \tau_n = \frac{\beta_{n+1} - \beta_n}{ \omega_{n+1} \gamma_g}
\ge \frac{\beta_{n+1} - \beta_n}{ \gamma_g}.
\een
Hence, we have from $\beta_n \ge c n^2$ that
$ S_N =\sum_{n=1}^N \beta_{n}\tau_{n} \ge  \frac{\beta_{N+1} - \beta_1}{ \gamma_g}
\ge c_1 N^2$ for some $c_1 > 0$.
Consequently, $G(X_N,Y_N) \le C_3/  N^{2}$ with $C_3 = \beta_1 A_1/c_1$ follows from (\ref{G-XY}).
Hence, property (c) holds. We complete the proof.
\end{proof}

\subsection{Linear convergence when $g$ and $f^*$ are strongly convex}
\label{sec-gf-strongly}
In this section, we assume that both $g$ and $f^*$ are strongly convex, with modulus $\gamma_g>0$ and $\gamma_f>0$, respectively, and
establish nonergodic linear convergence rate of GRPDA with linesearch. Note that when only one of the component function is strongly convex, say $g$, the strong convexity parameter $\gamma_g$ is used in the algorithm, see \eqref{beta-gstrong}. However,
when both $g$ and $f^*$ are strongly convex, the strong convexity modulus $\gamma_g$ and $\gamma_f$ do not need to be known in advance  since they are only used in the analysis and play no role in the algorithm itself.
We now summarize the modified algorithm as follows.
\vskip5mm
\hrule\vskip2mm
\begin{algo}
[GRPDA with linesearch when $g$ and $f^*$ are strongly convex]\label{algo-acc3-L}
{~}\vskip 1pt {\rm
\begin{description}
\item[{\em Step 0.}] Let $\psi_0 = 1.3247...$ be the unique real root of $\psi^3-\psi-1 = 0$. Choose $\psi\in(\psi_0,\phi)$, $\beta>0$, $\tau_0 > 0$ and $\mu\in(0,1)$. Choose $x_0=z_{0}\in \bR^q$ and $y_0\in \bR^p$. Set $\varphi=\frac{1+\psi}{\psi^2}$ and $n=1$.
\item[{\em Step 1.}]Compute
\be
z_{n}&=& \frac{\psi-1}{\psi} x_{n-1} + \frac{1}{\psi}z_{n-1}, \label{z-gfstrong} \\
x_{n}&=&\prox_{\tau_{n-1} g}(z_{n}-\tau_{n-1} K^\top y_{n-1}).\label{x-gfstrong}
\ee
\item[{\em Step 2.}]  Let $\tau=\varphi\tau_{n-1}$ and compute
\be\label{y-gfstrong}
   y_{n}=\prox_{\beta\tau_{n} f^*}(y_{n-1}+\beta\tau_{n} Kx_{n}),
\ee
where $\tau_n = \tau \mu^i$ and $i$ is the smallest nonnegative integer such that
\be\label{y-ls-gfstrong}
\sqrt{\beta\tau_{n}}\|K^\top y_{n}-K^\top y_{n-1}\|\leq \sqrt{\psi / \tau_{n-1}}\|y_{n}-y_{n-1}\|.
\ee
\item[{\em Step 3.}] Set $n\leftarrow n + 1$ and return to Step 1.
  \end{description}
}
\end{algo}
\vskip1mm\hrule\vskip5mm

With respect to Algorithm~\ref{algo-acc3-L}, we have the following key lemma.
\begin{lem}\label{lem-beta-fg-acc3}
  Let $\{\tau_n: n\geq 0\}$ be the sequence generated by Algorithm~\ref{algo-acc3-L}.
%  , and $\gamma_g>0$ and $\gamma_f>0$ be the strong convexity modulus of $g$ and $f^*$, respectively.
   Then, we have the following properties.
(i) The linesearch step of Algorithm~\ref{algo-acc-L}, i.e., Step 2, always terminates.
(ii)  There exists a constant $\theta>1$  such that for any $N \ge 1$, we have
$\Gamma_{N+1}: = \prod_{n=2}^{N+1} \theta_n \ge \theta^N$ with
\begin{equation}\label{def-theta-n}
\theta_n = \min\left\{1 + \omega_n\gamma_g\tau_{n-1},\, 1 + \beta\gamma_f\tau_{n-1}\right\},
\quad n \ge 2,
\end{equation}
where $\omega_n = \frac{\psi - \varphi}{\psi+\gamma_g \varphi\tau_{n-1}}$ is defined in \eqref{omega-gstrong}
\end{lem}
\begin{proof}
The conclusion (i) follows from the same proof of conclusion (i) in Lemma~\ref{lem_bound} with setting
 $\sigma =1$. We next prove conclusion (ii).  Given any integer $N \ge 1$,
it follows from the same proof  of conclusion (iii)  in Lemma~\ref{lem_bound} with setting $\sigma =1$
that $|\mathcal{K}_N| \ge \hat{c} N$ for some
constant $\hat{c} > 0$, where
\be\label{def_K_N}
\mathcal{K}_N = \{1 \le n \le N: \tau_n \ge \hat{\tau}\} \text{~~with~~}
\hat{\tau} = \frac{1}{L} \frac{\sqrt{\psi}}{\sqrt{\beta \varphi}} > 0
\ee
and $|\mathcal{K}_N|$ is the cardinality of the set $\mathcal{K}_N$.
Then, by the definition of $\Gamma_N$ and $\theta_n$, we have
\begin{eqnarray*}
\Gamma_{N+1} &= & \prod_{n=2}^{N+1} \theta_n = \prod_{n=2}^{N+1}
 \min\left\{1 + \frac{(\psi-\varphi)\gamma_g\tau_{n-1}}{\psi+\gamma_g \varphi\tau_{n-1}} ,\,
 1 + \beta\gamma_f\tau_{n-1}\right\}\\
&\ge & \prod_{n \in \mathcal{K}_N}
 \min\left\{1 + \frac{(\psi-\varphi)\gamma_g \tau_n}
 {\psi+\gamma_g \varphi \tau_n} ,\,  1 + \beta\gamma_f\tau_n\right\}\\
 &\geq& \tilde{\theta}^{|\mathcal{K}_N|} \ge \tilde{\theta}^{\hat{c} N} = \theta^N,
\end{eqnarray*}
where $\theta = \tilde{\theta}^{\hat{c}} >1$ and
$\tilde{\theta} = \min\big\{1 + \frac{(\psi-\varphi)\gamma_g \hat{\tau}}
{\psi+\gamma_g \varphi \hat{\tau}} ,\,  1 + \beta\gamma_f \hat{\tau}\big\} >1$.
\end{proof}

Based on   Lemma~\ref{lem-beta-fg-acc3}, we have the following iterative and nonergodic linear convergence results.
\begin{thm}[Linear convergence]
 Let $\{(z_n,x_n,y_n): n\geq 1\}$ be the sequence generated by Algorithm~\ref{algo-acc3-L}.
Then, the following holds:\\
(a) $\{(z_n,y_n): n\geq 1\}$ converge to the unique primal and dual optimal solutions $({\bar x}, {\bar y})$; \\
(b) there exit constants $C_1, C_2>0$ and $\theta >1$ such that for all $n\geq 1$ there hold
$\|z_{n+2}-\bar{x}\| \leq C_1/\theta^n$ and $\|y_{n}-\bar{y}\| \leq C_2/\theta^n$; \\
%(c) there exists a constant $C_3 >0$ such that $0 \le G(x_n, y_n) \le C_3/\theta^n$.
(c) there exist constants $C_3 >0$ and $\hat{\theta} > 1$ such that
$\min\limits_{1 \le n \le N} G(x_n, y_n) \le C_3/\hat{\theta}^N$.
\end{thm}
\begin{proof}
Following the line of proof of Theorem \ref{thm:gstrong}, similar to \eqref{stronger-c},
one can easily show that
\ben%\label{ieq-fgstrong}
\tau_n G(x_n,y_n)
&\leq&\langle x_{n+1}-z_{n+1}, \bar{x}-x_{n+1}\rangle+
\frac{1}{\beta} \big\langle y_{n}-y_{n-1}, \bar{y}-y_{n}\big\rangle
+\psi\delta_n \big\langle x_{n}-z_{n+1}, x_{n+1}- x_n\big\rangle \nonumber\\
&&+ \tau_n \langle K^\top y_{n}- K^\top y_{n-1}, x_n-x_{n+1}\rangle- \frac{\gamma_g\tau_n}{2} \|x_{n+1}-\bar{x}\|^2 - \frac{\gamma_f\tau_n}{2}  \|y_{n}-\bar{y}\|^2,
\een
where $\delta_n = \tau_n/\tau_{n-1}$.
Then, by the same arguments as in \eqref{jy-0316-1}-\eqref{ineq13}, we have
\be
&&(1+\gamma_g\tau_n)\frac{\psi}{\psi-1}\|z_{n+2}-\bar{x}\|^2
+\big(1/\beta +\gamma_f\tau_n\big) \|y_{n}-\bar{y}\|^2+ 2 \tau_n G(x_n,y_n)\nonumber\\
&\leq&\frac{\psi+\gamma_g\tau_n}{\psi-1}\|z_{n+1}-\bar{x}\|^2+ \frac{1}{\beta}\|y_{n-1}-\bar{y}\|^2.
\label{jy-0316-3}
\ee
From  \eqref{relax-coe} and the definition of $\theta_{n+1}$, it holds that
\[\label{jy-0316-2}
\frac{\psi(1+\gamma_g\tau_n)}{\psi-1} = \frac{\psi(1+\gamma_g\tau_n)}{\psi+\gamma_g\tau_{n+1}} \frac{\psi+\gamma_g\tau_{n+1}}{\psi-1}
\geq   \big(1 + \omega_{n+1}\gamma_g\tau_n\big)  \frac{\psi+\gamma_g\tau_{n+1}}{\psi-1}
\geq   \theta_{n+1}   \frac{\psi+\gamma_g\tau_{n+1}}{\psi-1},
\]
where $\omega_n$ and $\theta_n$ are defined in \eqref{omega-gstrong} and \eqref{def-theta-n}, respectively.
Let
\[\label{jy-def:An}
A_n := \frac{\psi+\gamma_g\tau_n}{\psi-1}\|z_{n+1}-\bar{x}\|^2+ \frac{1}{\beta}\|y_{n-1}-\bar{y}\|^2.
\]
Then, we have from \eqref{jy-0316-3}, \eqref{jy-0316-2}
and $G(x_n,y_n) \ge 0$ that
\be\label{ineq-gf}
\theta_{n+1} A_{n+1} \le \theta_{n+1} A_{n+1} + 2\tau_n G(x_n,y_n) \leq A_n,
\ee
which implies $ \big(\prod_{n=2}^{N+1} \theta_n \big) A_{N+1} \le A_1$.
By (ii) of Lemma~\ref{lem-beta-fg-acc3}, we have
$ \Gamma_{N+1} =  \prod_{n=2}^{N+1} \theta_n \ge \theta^N$ for some $ \theta > 1$.
Hence, we have $A_{N+1} \le A_1 / \theta^{N}$, which implies conclusions (a) and (b).
%Finally, the conclusion (c) follows from (b) and the locally Lipschitz continuity of $G(\cdot, \cdot)$. \comm{ Need to make sure this final step.}
Finally, let $\mathcal{K}_N = \{1 \le n \le N: \tau_n \ge \hat{\tau} >0\}$ be defined in
\eqref{def_K_N} and $\hat{N} = \max \{n : n \in \mathcal{K}_N \}$. Then, we have
$\hat{N} \ge |\mathcal{K}_N| \ge \hat{c} N$ for some $\hat{c}>0$. Thus, it follows from (b) and \eqref{ineq-gf} that
\ben
\min_{1 \le n \le N} G(x_n, y_n) \le G(x_{\hat{N}}, y_{\hat{N}})
\le \frac{A_{\hat{N}}}{2 \hat{\tau}}
\le  \frac{A_1}{2 \hat{\tau} \theta^{\hat{N}-1}}
\le \frac{A_1}{2 \hat{\tau} \theta^{\hat{c} N-1}} := \frac{C_3}{\hat{\theta}^N},
\een
where $C_3 = \theta A_1/(2 \hat{\tau})$ and $\hat{\theta} = \theta^{\hat{c}} > 1$.
So, (c) holds. We complete the proof.
\end{proof}

%
%\comm{An ergodic result is as follows:}
%\minew{
%Since $G(x_n,y_n)$ is nonnegative, multiplying both sides of $2\tau_n G(x_n,y_n) \leq A_n-\theta_{n+1} A_{n+1}$ by $\Theta_n := \prod_{j=1}^n\theta_j$ yields
%\ben
%2\Theta_n\tau_n G(x_n,y_n) \leq \Theta_n A_n - \Theta_{n+1}  A_{n+1}.
%\een
%Summing over $n = 1,\ldots,N$ and using the convexity of $G(\cdot,\cdot)$, we obtain
%\be\label{G-add}
% G(X_N,Y_N) \leq  {1\over 2S_N} \sum_{n=1}^N \Theta_n\tau_n G(x_n,y_n) \leq {\theta_1 A_1\over 2 S_N},
%\ee
%where
%\ben
%  S_N= \sum_{n=1}^N\Theta_n\tau_n,
%~~X_N= \frac{\sum_{n=1}^N\Theta_n\tau_nx^n}{S_N},
%~~Y_N= \frac{\sum_{n=1}^N\Theta_n\tau_ny^n}{S_N}.
%\een
%Let $\mathcal{K}_N = \{1 \le n \le N: \tau_n \ge \hat{\tau}\} = \{n_1, n_2, \ldots, n_{|\mathcal{K}_N|}\}$, where $1\leq n_1 < n_2 < \ldots < n_{|\mathcal{K}_N|} \leq N$. Then,
%for $1\leq i\leq  |\mathcal{K}_N|$, we have $\mathcal{K}_{n_i} = \{n_1, n_2, \ldots, n_i\}$ and $|\mathcal{K}_{n_i}|=i$.
%Therefore, we have from the proof of Lemma \ref{lem-beta-fg-acc3} that
%\ben
%S_N
%\geq  \sum_{n_i\in \mathcal{K}_N} \prod_{j\in \mathcal{K}_{n_i} }\theta_j\tau_{n_i}
%\geq \hat{\tau} \sum_{n_i\in \mathcal{K}_N} \tilde{\theta}^i
%=   \hat{\tau}   \sum_{i=1}^{|\mathcal{K}_N|} \tilde{\theta}^i
%%
%= \frac{ \hat{\tau} \tilde{\theta}(\tilde{\theta}^{|\mathcal{K}_N|}-1)}{\tilde{\theta}-1}
%\geq \hat{\tau} \tilde{\theta}^{|\mathcal{K}_N|}
%\geq \hat{\tau} \theta^N,
%\een
%where $\theta>1$, which together with \eqref{G-add} gives
%\ben
% G(X_N,Y_N)  \leq
%{\theta_1 A_1\over 2\hat{\tau} \theta^N} .
%\een
%}

\section{Numerical results}
\label{sec-experiments}
In this section, numerical results are presented to demonstrate the performance of the proposed  Algorithm \ref{algo2} (GRPDA-L) and Algorithm \ref{algo-acc-L} (AGRPDA-L). We set  $\sigma = 0.99$ for GRPDA-L and $\psi = 1.5$ and $\mu = 0.7$ for both algorithms. First, choose $y_{-1}$ arbitrarily in a small neighborhood of the starting point $y_0$
such that $K^\top y_{-1}\neq K^\top y_0$ and compute
\be\label{t0tb}
m=\frac{\|y_{-1}-y_0\|}{\|K^\top y_{-1}-K^\top y_0\|} \geq {1\over L}.
\ee
%Obviously, it holds that $m\geq1/L$.
Then, set $\tau_0=\frac{\sqrt{\psi}}{\sqrt{\beta}}m\geq\frac{\sqrt{\psi}}{\sqrt{\beta}L}$ for GRPDA-L and $\tau_0=\frac{\sqrt{\psi}}{\sqrt{\beta_0}}m\geq\frac{\sqrt{\psi}}{\sqrt{\beta_0}L}$ with some $\beta_0>0$ for AGRPDA-L.
We compare the proposed algorithms with their corresponding counterparts without linesearch, i.e.,
GRPDA \cite[Algoirthm 3.1]{ChY2020Golden} with $\tau=\frac{\sqrt{\psi}}{\sqrt{\beta}L}$ and $\psi = 1.618$, as well as  the state-of-the-art primal-dual algorithm with linesearch PDA-L, i.e., \cite[Algorithm 1]{Malitsky2018A}, with $\mu = 0.7$,  $\delta= 0.99$ and $\tau_0=\|y_{-1}-y_0\|/(\sqrt{\beta}\|K^\top y_{-1}-K^\top y_0\|)$.
Other parameters will be specified in the following.

All experiments were performed within Python 3.8 on an Intel(R) Core(TM) i5-4590 CPU 3.30GHz PC with 8GB of RAM running on 64-bit Windows operating system. For reproducible purpose, the codes are provided at \url{https://github.com/cxk9369010/GRPDA-Linesearch}. We solve the minimax matrix game problem and the LASSO problem for comparison.

\begin{prob}[Minimax matrix game]\label{pro_mm}
The minimax matrix game problem is given by
\be\label{mm_pro}
\min_{x \in \D_q}\max_{y\in \D_p} \lr{Kx, y},
\ee
where $K\in \R^{p\times q}$, $\Delta_q = \{x\in \R^q: \sum_i x_i = 1, \, x\geq 0\}$ and $\Delta_p = \{y\in \R^p: \sum_i y_i = 1, \, y\geq 0\}$
denote the standard unit simplex in $\R^q$ and $\R^p$, respectively.
\end{prob}

Clearly, \eqref{mm_pro} is a special case of \eqref{saddle_point} with $g = \iota_{\D_q}$ and $f^* = \iota_{\D_p}$, where
$\iota_C$ denotes the indicator function of a set $C$.  Since neither $g$ nor $f^*$ is strongly convex, only the non-accelerated algorithms GRPDA, GRPDA-L and PDA-L are relevant here.
For $(x, y)\in \D_q\times \D_p$, the primal-dual gap function is given by $G(x, y) := \max_i(Kx)_i - \min_j(K^\top y)_j$.
Initial points for all the algorithms are set to be $x_0 =\frac{1}{q}(1,\ldots,1)^\top\in\R^q$ and $y_0 =\frac{1}{p}(1,\ldots,1)^\top\in\R^p$.
The projection onto the unit simplex is computed by the algorithm from \cite{Duchi2011Diagonal}. We set $\psi=1.618$ and $\tau= \sigma= 1/\|K\|$
for GRPDA and $\beta= 1$ for PDA-L and GRPDA-L.
As in \cite{Malitsky2018A}, we generated $K\in \bR^{p\times q}$ randomly in the following four different ways with random number generator $seed=50$:
\bi
\item[(i)] All entries of $K$ were generated independently from the uniform distribution in $[-1, 1]$, and $(p,q) = (100,100)$;
\item[(ii)] All entries of $K$ were generated independently from the normal distribution $\cN (0,1)$,  and $(p,q) = (100,100)$;
\item[(iii)] All entries of $K$ were generated independently from the normal distribution $\cN (0,10)$,  and $(p,q) = (500,100)$;
\item[(iv)] The matrix $K$ is sparse with 10\% nonzero elements generated independently from the uniform distribution in $[0,1]$,  and $(p,q) = (1000,2000)$.
\ei

For a given $\epsilon>0$, we terminate the algorithms when $G(x_n, y_n)<\epsilon$ or $n = n_{\max}$, where $n_{\max}$ is the maximum number of iterations allowed. In this section, we set $n_{\max}=3\times 10^5$ and examine how the values of the primal-dual gap function $G(x,y)$ decrease as CPU time proceeds.
Table \ref{table1} presents the total CPU time (Time, in seconds), the number of iterations (Iter) and the number of extra linesearch trial steps (\#LS) of PDA-L and GRPDA-L as compared with their counterparts without linesearch. We emphasize that for each trial of linesearch a projection onto the unit simplex is required for this example.
The decreasing behavior of the primal-dual gap function values (abbreviated as PD gap $G(x_n,y_n)$) versus CPU time is shown in Figure \ref{Fig 5} for the compared algorithms  with $\epsilon=10^{-10}$.

\begin{table}[htpb]
% The table caption is above the table.
\caption{Results of GRPDA, PDA-L and GRPDA-L on problem \ref{pro_mm}.  In the table, ``---" represents that the algorithm reached the maximum number of iterations without satisfying the stopping condition.}\label{table1}
\vspace{.1 in}
\center
\small
\begin{tabular}{|c|c|cc|ccc|ccc|}
\hline
\multirow{2}{*}{$\epsilon$}&\multirow{2}{*}{Test}&\multicolumn{2}{|c|}{GRPDA}&\multicolumn{3}{|c|}{PDA-L}&\multicolumn{3}{|c|}{GRPDA-L}\\
                                                    &&Iter &Time &Iter& \#LS &Time	&Iter& \#LS &Time\\
\hline
\multirow{4}{1cm}{$10^{-7}$}&(i)	&25688  &3.7  &18612  &18376 	&4.5  &11010  &3250   & 2.1	 	\\
&(ii)	&103788  &14.9  &40676  &40209 	&10.2  &32656  &9646   &7.0 	 	\\
&(iii)	& ---  &79.7  &73197  &72903 	&32.4  &64628  &19088   & 23.4	 	\\
&(iv)	&---  &486.5  &45705  &44633 	&118.7  &30356  &8961   &65.5 	 	\\
\hline
\multirow{4}{1cm}{$10^{-10}$}&(i)	&151134 &20.2 &58282 &57561	&16.0 &45645 &13481  &9.8	 	\\
&(ii)	&245612 &33.5 &89644 &88622	&24.4 &75467 &22292  &15.9	 	\\
&(iii)	&--- &79.9 &155281 &154655	&69.3 &145527 &42985  &51.2	 	\\
&(iv)	&--- &486.5 &---&292982	&699.6&--- &88613 &	617.7 	\\
\hline
\end{tabular}
\end{table}

It can be seen from Table \ref{table1} that PDA-L requires approximately one extra linesearch trial step per outer iteration, while
GRPDA-L requires roughly one extra linesearch trial step per three outer iterations. As a result, GRPDA-L consumed less CPU time than PDA-L.
From the results in Figure \ref{Fig 5}, for all the four tests GRPDA-L performs the best, followed by PDA-L, both are faster than GRPDA. It can be seen from Figure \ref{Fig 5} that test (iv) is a difficult case, which makes all the compared algorithms fail to reduce the primal-dual gap function value to less than $10^{-10}$ within the prescribed maximum number of iterations.
%
%\revise{To reduce oscillation and make clear comparison, we draw the lines by every 300 time steps. }

\begin{figure}[htp]
\centering
\subfigure[Test (i).]{
\includegraphics[width=0.45\textwidth]{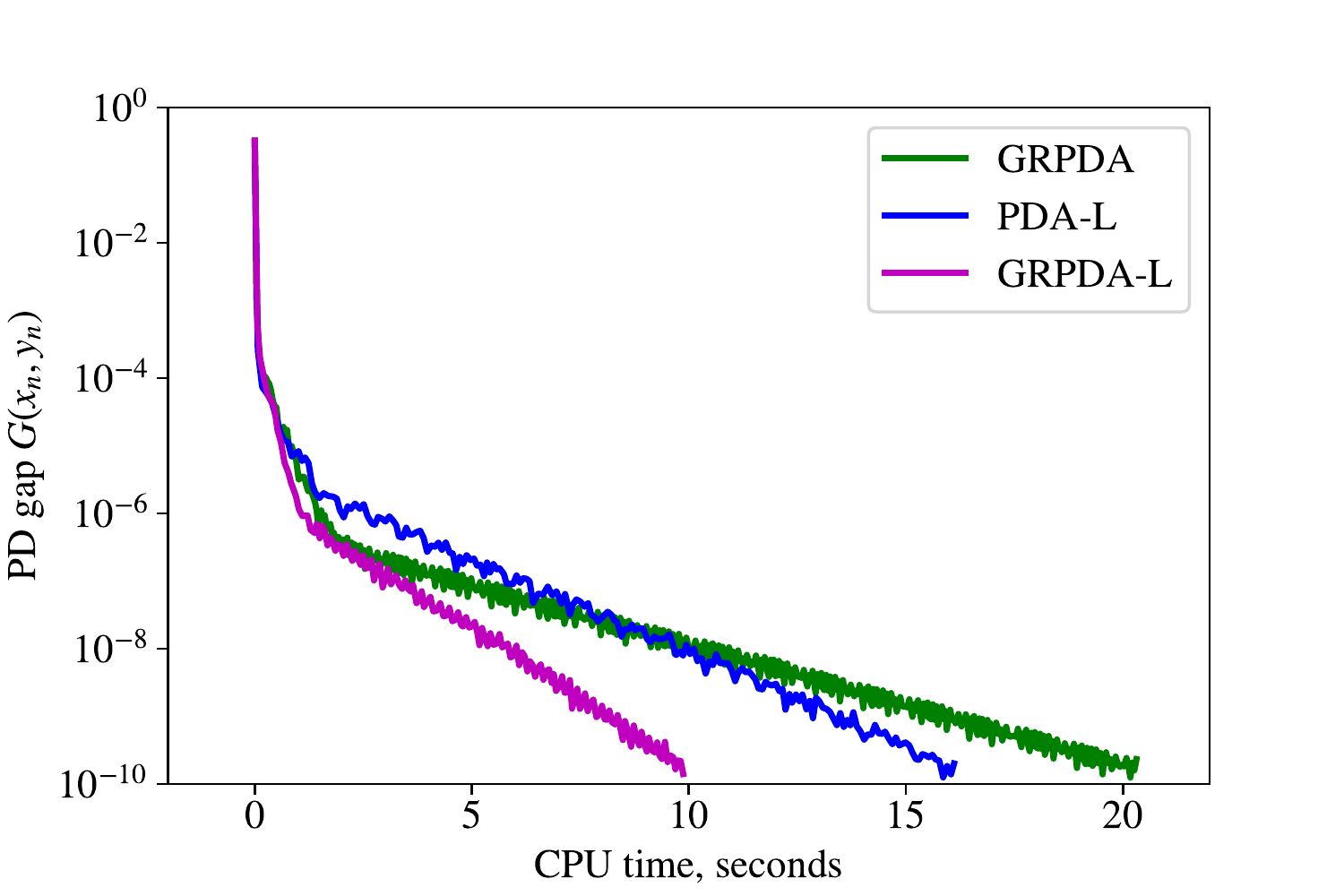}}
\subfigure[Test (ii).]{
\includegraphics[width=0.45\textwidth]{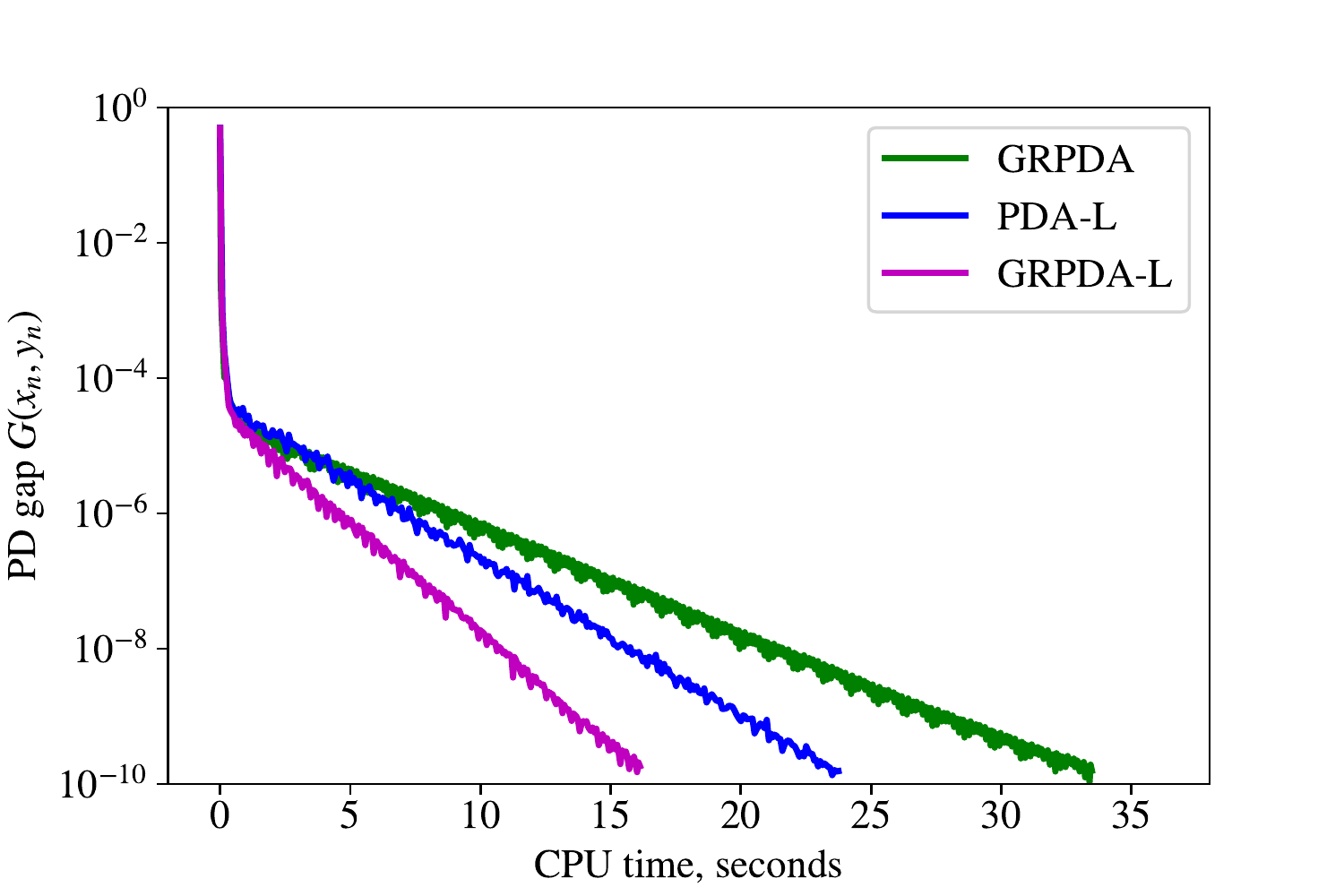}}
\subfigure[Test (iii).]{
\includegraphics[width=0.45\textwidth]{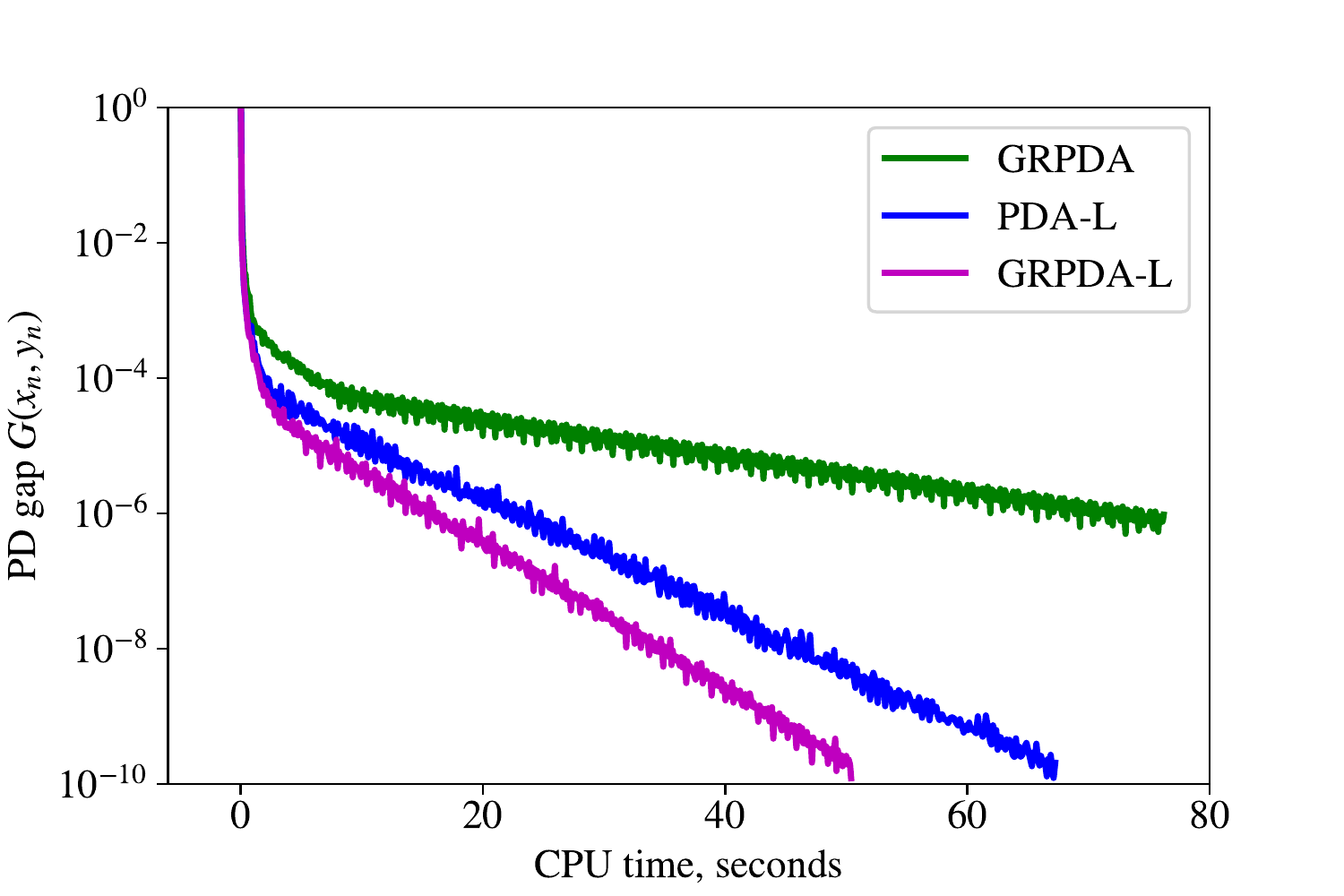}}
\subfigure[Test (iv).]{
\includegraphics[width=0.45\textwidth]{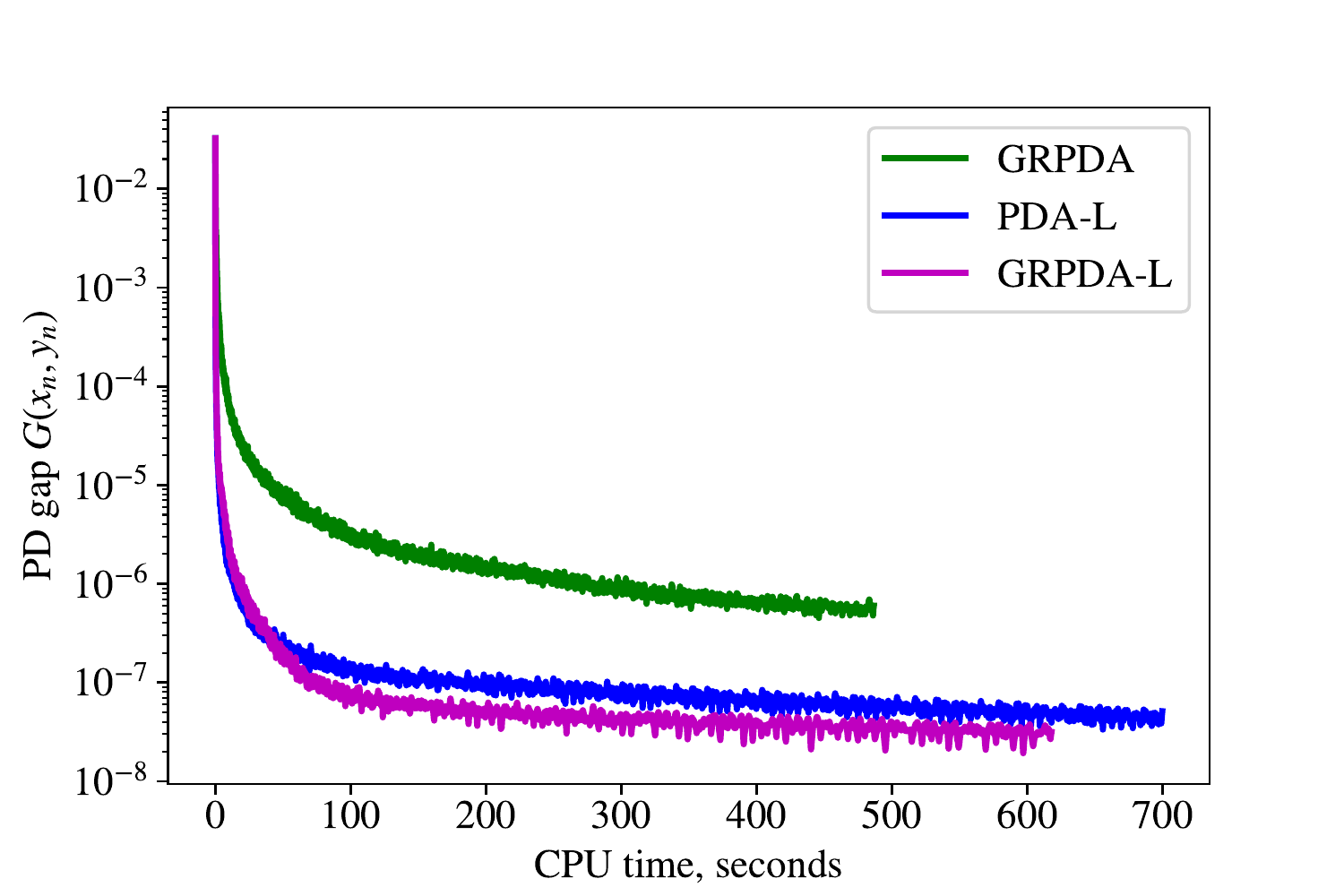}}
\caption{
Comparison results of GRPDA, GRPDA-L and PDA-L on Problem \ref{pro_mm}: PD gap versus CPU time.
}
\label{Fig 5} %% label for entire figure
\end{figure}

\begin{prob}[LASSO]\label{pro_1}
Let $K\in \R^{p\times q}$ be a sensing matrix and $b\in \R^p$ be an observation vector.
One form of the LASSO problem is to recover a sparse signal via solving
\begin{equation}
  \label{pro_lasso1}
\min_x F(x):= \mu  \|x\|_1 + \frac 1 2 \|Kx-b\|^2,
\end{equation}
where $\mu > 0$ is a regularization parameter.
\end{prob}

%We can rewrite the problem above in a primal-dual form as follows:
%\be\label{pro_lasso1}
%\min_{x\in \R^q} \max_{y\in \R^p} g(x)+\langle Kx,y\rangle-f^*(y),
%\ee
It is easy to verify that the LASSO problem \eqref{pro_lasso1} can be represented as the saddle point problem \eqref{saddle_point} with
%
%where $f(p) = \frac 1 2 \|p-b\|^2$,
%
$g(x) = \mu \|x\|_1$ and $f^*(y) = \frac 1 2  \|y\|^2 + \langle b, y\rangle$.
Thus, the proximal operator $\prox_{\tau f^*}(\cdot)$ is linear. In fact, it holds that
\ben
y_{n} = \prox_{\beta_n\tau_n f^*}(y_{n-1}+\beta_n\tau_n K x_{n})=\frac{y_{n-1}+\beta_n\tau_n(K x_{n}-b)}{1+\beta_n\tau_n}.
\een
Therefore,   there is no extra matrix-vector multiplications introduced  within a linesearch step for GRPDA-L as
$K^\top y_{n}$ can always be obtained via a convex combination of the already computed quantities $K^\top y_{n-1}$ and $K^\top(K x_{n}- b)$.
On the other hand,  problem \eqref{saddle_point} is equivalent to
\ben%\label{}
\max_{x\in \R^q}\min_{y\in \R^p} ~~f^*(y) + \langle -K^\top y, x \rangle  -g(x).
\een
Then, by swapping ``$\max_{x\in \R^q}$" with ''$\min_{y\in \R^p}$" and $(g,K,x,q)$ with $(f^*,-K^T,y,p)$, the strong convexity of $\frac 1 2 \|y\|^2 + \langle b,y\rangle$
(previously $f^*$) can be transferred to $g$, which enables the application of the accelerated version, i.e., AGRPDA-L.  Therefore, the algorithms to compare in this experiment are GRPDA-L, AGRPDA-L and PDA-L (i.e., \cite[Algorithm 1]{Malitsky2018A}).
GRPDA without linesearch will not be compared since it is the most inefficient.

We set $seed=100$ and generate a random vector $x^*\in \bR^q$ for which $s$ random coordinates are drawn from $\cN(0,1)$ and the rest are set to be zero. Then, we generate $\omega\in \bR^p$ with entries drawn from $\cN(0, 0.1)$ and set $b = Kx^* +\omega$.
The matrix $K\in \bR^{p\times q}$ is constructed in the following ways:
\bi
\item[(i)] All entries of $K$ are generated independently from $\cN(0, 1)$. The $s$ entries of $x^*$ are drawn from the uniform distribution in $[-10, 10]$;
\item[(ii)]
    First, we generate a matrix $A\in\R^{p\times q}$, whose entries are independently drawn from $\cN(0,1)$. Then, for a scalar $v\in(0, 1)$ we construct the matrix $K$ column by column as follows: $K_1 =  A_1/\sqrt{1-v^2}$ and $K_j = vK_{j-1} + A_j$, $j=2,\ldots,q$.
    Here $K_j$ and $A_j$ represent the $j$th column of $K$ and $A$, respectively.
     As $v \in (0,1)$ becomes larger, $K$ becomes more ill-conditioned.
     In this experiment we take $v = 0.5$ and $v=0.9$, respectively.
     The sparse vector $x^*$ is generated in the same way as in case (i).
\ei

%The algorithmic parameters are selected in a way  that is similar to . Specifically,

In both cases, the regularization parameter $\mu$ was set to be $0.1$.
Similar to \cite{Malitsky2018A}, we set $\beta= 400$ for PDA-L and GRPDA-L. For AGRPDA-L, we set $\gamma=0.01$ and $\beta_0=1$ as in \cite{ChY2020Golden}.
The initial points for all algorithms are $x_0 = (0,\ldots, 0)^\top$ and $y_0=Kx_0-b$.

In this experiment, we first ran all the algorithms by a sufficiently large number of iterations and then chose the minimum attainable function value as an approximation of the optimal value $F^*$ of \eqref{pro_lasso1}.
Again, for a given $\epsilon>0$, we terminate the algorithms when $F(x_n)-F^*<\epsilon$ or $n = n_{\max}$. In this experiment, we set $\epsilon=10^{-12}$ and $n_{\max}=8\times 10^4$ to examine  their convergence behavior.

\begin{table}[htpb]
% The table caption is above the table.
\caption{Results of GRPDA-L, PDA-L and AGRPDA-L on the LASSO problem \eqref{pro_lasso1}.}\label{table2}
\vspace{.1 in}
\center
\small
\begin{tabular}{|c|c|ccc|ccc|ccc|}
\hline
\multirow{2}{*}{$\epsilon$}&\multirow{2}{*}{Test}&\multicolumn{3}{|c|}{PDA-L} &\multicolumn{3}{|c|}{GRPDA-L}&\multicolumn{3}{|c|}{AGRPDA-L}\\
                                                    &&Iter & \#LS &Time &Iter& \#LS &Time	&Iter& \#LS &Time\\
\hline
\multirow{3}{0.8cm}{$10^{-8}$}&(i) 	& 4757  &  4688 & 14.7  & 4043 	& 1186  & 12.5  & 2450   & 723 &	10.9 	\\
&(ii)  $v=0.5$		& 6167  &  6130 & 21.3  & 5213 	& 1532  & 17.1  &  1759  & 517 &8.1	 	\\
&(ii)  $v=0.9$		& 27899  & 27889  & 94.2  &  26080	&  7697 & 86.7  & 7480   & 2208 &	34.5 	\\
\hline
\multirow{3}{0.8cm}{$10^{-12}$}&(i) 	& 11234  & 11082  & 34.8  &  9287	& 2735  & 28.8  &  3539  & 1043 &	15.8 	\\
&(ii) $v=0.5$	&  14913 &  14831 & 49.1  & 12330 	& 3634  &  40.3 &  3124  & 922 &14.3	 	\\
&(ii) $v=0.9$		& 64003  & 64007  & 213.8  &  55758	&  16464 & 183.1  &  12216  & 3608 &	55.6 	\\
\hline
\end{tabular}
\end{table}

It can be seen from the results in Table \ref{table2} that similar conclusion can be drawn, i.e., PDA-L requires approximately one extra linesearch trial step per outer iteration, while
GRPDA-L and AGRPDA-L require roughly one extra linesearch trial step per three outer iterations. Since the proximal operator $\prox_{\tau f^*}(\cdot)$ is linear and does not incur extra computations, GRPDA-L and PDA-L perform similarly in terms of outer iteration and CPU time.
In comparison, AGRPDA-L performs the best, i.e., takes much less number of iterations and CPU time.
The evolution of  function value residuals $F(x_n)-F^*$ versus CPU time is given in Figure \ref{Fig:lasso}, from which
it can be seen that AGRPDA-L, which takes advantage of strong convexity, is much faster than GRPDA-L and PDA-L, which do not.
These results lead to the conclusion that strong convexity of the component functions, if properly explored, helps to improve the performance of primal-dual type algorithms.

\begin{figure}[htp]
\centering
\subfigure[Case (i).]{
\includegraphics[width=0.32\textwidth]{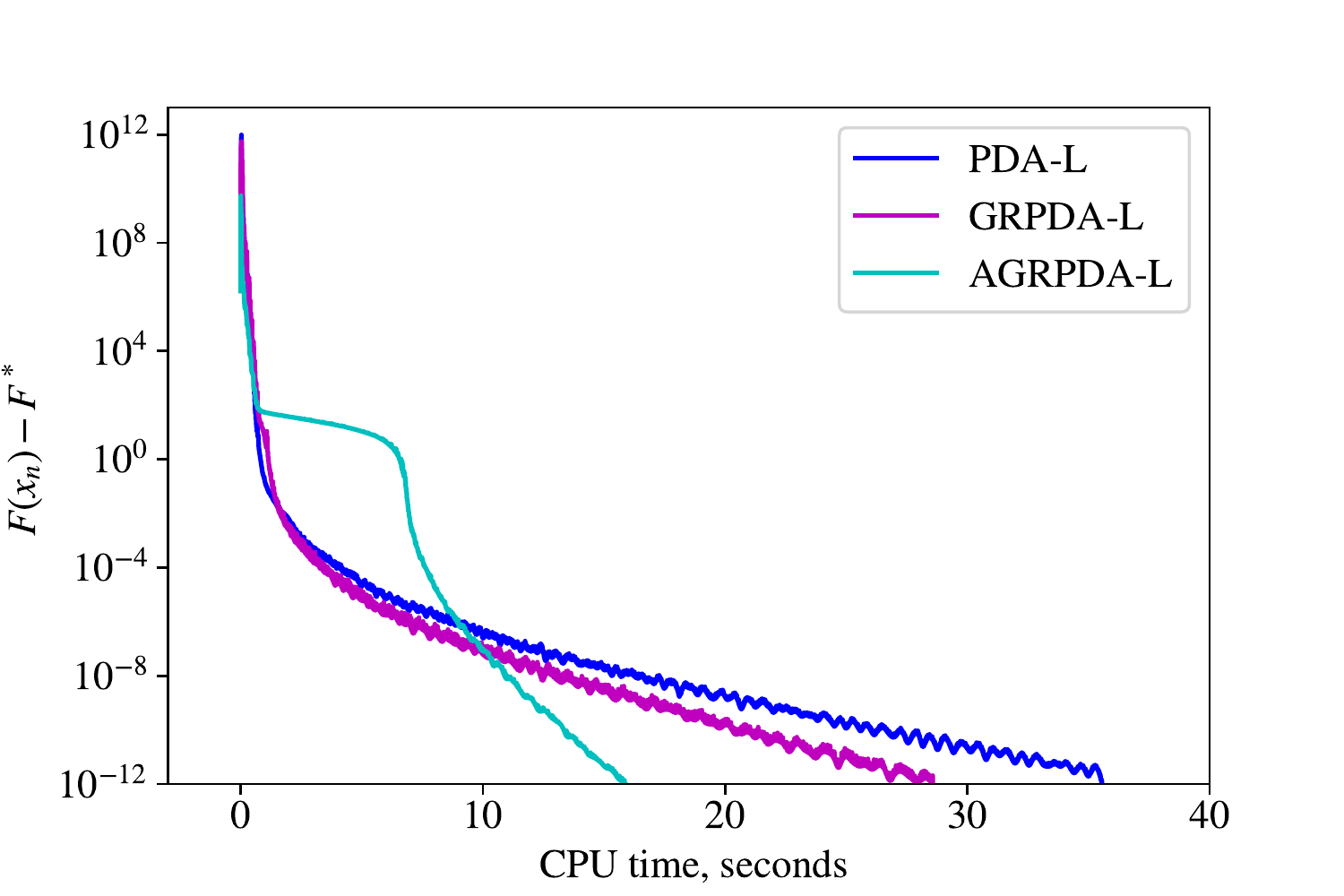}}
\subfigure[Case (ii) with $v = 0.5$.]{
\includegraphics[width=0.32\textwidth]{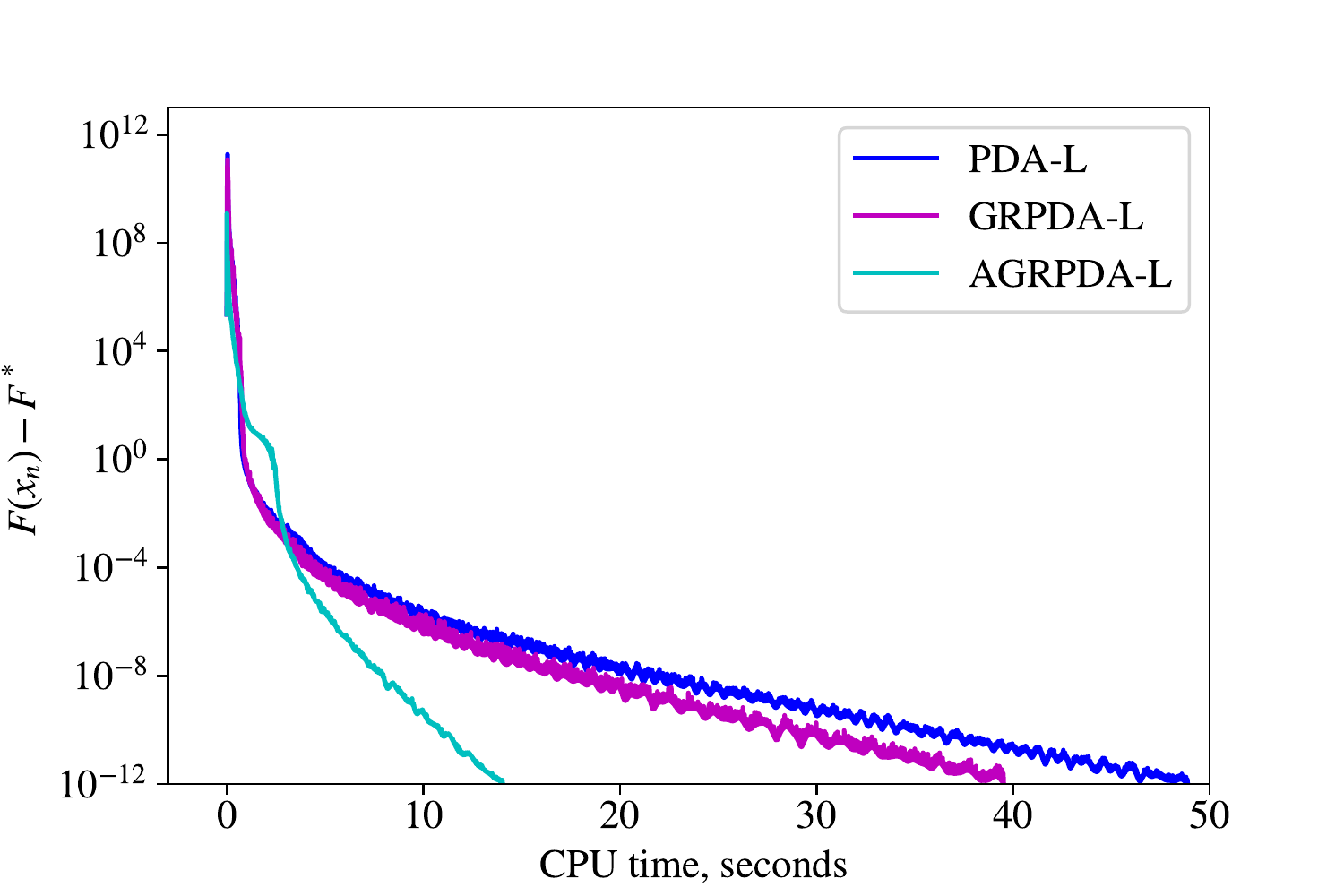}}
\subfigure[Case (ii) with $v = 0.9$.]{
\includegraphics[width=0.32\textwidth]{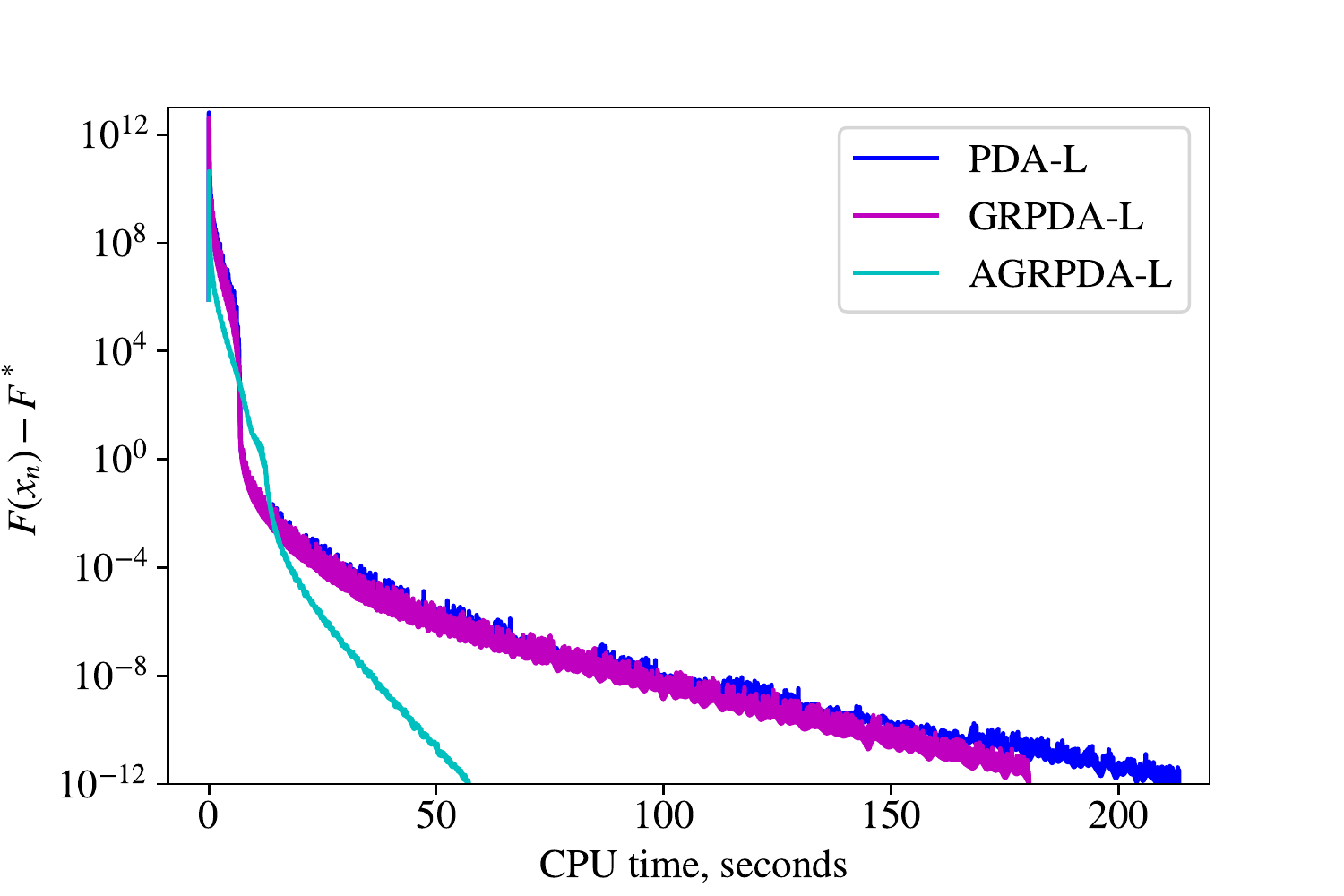}}
\caption{Experimental results on the LASSO problem. Decreasing behavior of function value versus CPU time.
From left to right: Case (i) with $(p,q,s) = (1000,2000,100)$, Case (ii) with $(p,q,s,v) = (1000,2000,10,0.5)$, and Case (ii) with $(p,q,s,v) = (1000,2000,10,0.9)$.
} %Comparison of $F(x) - F^*$ for solving problem \ref{pro_1} with $p=200$, $q=1000$ and $s=10$. }
\label{Fig:lasso} %% label for entire figure
\end{figure}

\section{Conclusions}
\label{sec-conclusion}
In this paper, we have incorporated linesearch strategy into the golden ratio primal-dual algorithm (GRPDA) recently proposed  in \cite{ChY2020Golden}.
Global convergence and $\cO(1/N)$ ergodic convergence rate measured by primal-dual function gap
 are established in the general convex case.
When either one of the component functions is strongly convex, accelerated GRPDA with linesearch is proposed, which achieves $\cO(1/N^2)$ ergodic rate of convergence. Furthermore, when both component functions are strongly convex, nonergodic linear convergence results are obtained.
The proposed linesearch strategy does not require to evaluate the spectral norm of $K$ and adopts potentially much larger stepsizes.
In cases such as regularized least-squares problem, the proposed linesearch strategy
only requires minimal extra computational cost and thus is particularly useful.
Our numerical experimental on minimax matrix game and LASSO problems demonstrate
 the benefits gained by taking advantage of strong convexity and incorporating our proposed linesearch.
Experimentally, the extra linesearch trial steps used by golden ratio type primal-dual algorithms are about one-third of those proposed by Malitsky \cite{Malitsky2018A} and larger stepsizes can be accepted,
which could be significant  when the evaluations of proximal operators are nontrivial.

\section*{Acknowledgements}
Xiaokai Chang was supported by the Innovation Ability Improvement Project of Gansu (Grant No. 2020A022) and the Hongliu Foundation of Firstclass Disciplines of Lanzhou University of Technology, China.
Junfeng Yang was supported by the National Natural Science Foundation of China (NSFC grants 11922111 and 11771208).
Hongchao Zhang was supported by the USA National Science Foundation under grant DMS-1819161.

\bibliographystyle{abbrv}
%\bibliography{cxktex}

\begin{appendix}
\section{Proof of Lemma \ref{lem_bound}}\label{proof:lem-bound}

\begin{proof}
(i) Since \eqref{y-ls} is fulfilled whenever $\tau_n$ satisfies $\sqrt{\beta\tau_{n}} L \leq \sigma\sqrt{\psi / \tau_{n-1}}$, we have from $\tau_n = \varphi \tau_{n-1} \mu^i$ that \eqref{y-ls} is fulfilled
whenever $\sqrt{\mu^i } \leq \underline{\tau} / \tau_{n-1} $.
Hence, \eqref{y-ls} will be fulfilled by the linesearch procedure in Step 2 since $\mu\in(0,1)$.

(ii) We consider two cases. \\
Case 1: There exists a $\bar{k}$ such that $\tau_n \ge \underline{\tau}$ for all $n \ge \bar{k}$.
If there is no infinite subsequence $\{n_k: k \ge 1\} \subseteq \{1,2, \ldots\}$ such that
$\delta_{n_k} \ge \rho$, we have $\delta_n < \rho < 1$ for all $n$ sufficiently large.
Then, we will have from $\tau_n = \tau_0 \prod_{i= 1}^n \delta_i$ and $\rho < 1$ that
$\lim_{n \to \infty} \tau_n = 0$,
which contradicts with $\tau_n \ge \underline{\tau}$ for all $n \ge \bar{k}$
Hence, in this case, property (ii) holds. \\
Case 2:  There exists an infinite subsequence $\{n_i: i \ge 1\}$ such that
$\tau_{n_i} < \underline{\tau}$.
By the linesearch procedure in Step 2, for any $\tau_n \le \underline{\tau}$,
the initial trial $\tau_{n+1} = \varphi \tau_n$ will satisfy \eqref{y-ls} and be accepted
by the linesearch. Hence,  defining $\ell(t) = \lfloor \log_{\varphi} (\underline{\tau}/t) \rfloor$
where $t < \underline{\tau}$, we have the following property:
\begin{eqnarray}\label{property-A}
& \mbox{If $\tau_n < \underline{\tau}$, then with $k':=n+\ell(\tau_n)$ we have} & \nonumber \\
& \mbox{$\tau_k <  \underline{\tau}$ and
$\tau_{k+1} =\varphi \tau_k$ for all $k=n, n+1, \ldots, k'$, and $ \underline{\tau} \le \tau_{k'+1}
< \varphi \underline{\tau}$}.&
\end{eqnarray}
Here, $\lfloor t \rfloor$ is the largest integer less or equal to $t$.
So, for any  $\tau_{n_i} < \underline{\tau}$, we have
$\tau_{k'+1} \ge \underline{\tau}$, where $k' = n_i + \ell(\tau_{n_i})$.
In addition, we have
$\delta_{k'+1} = \tau_{k'+1}/\tau_{k'} = \varphi > \rho$.
Hence, in this case property (ii) also holds.

%
%it follows from $\tau_{n_i} < \underline{\tau}$ that
% there must exists a $ k' > n_i$ ($k'$ depending on $n_i$)
%such that  $\tau_{k+1} =\varphi \tau_k$ for all $k=n_i, n_i+1, \ldots, k'$ and
%$\tau_{k'+1} \ge \underline{\tau}$. In addition, we have
%$\delta_{k'+1} = \tau_{k'+1}/\tau_{k'} = \varphi > \rho$.
%Hence, in this case property (ii) also holds.

(iii) First, if $\tau_0 < \underline{\tau}$, by property \eqref{property-A}, we have
 $\tau_s = \tau_0 \varphi^s \ge \underline{\tau}$,
where $s = \ell(\tau_0) + 1$.
 Hence, without losing of generality, to show property (iii),
we can simply assume $\tau_1 \ge \underline{\tau}$.

Now, we show the following property:
\begin{equation}\label{property-B}
\mbox{ For any $\tau_{n-1} \ge \underline{\tau}$ and  $\tau_{n} < \underline{\tau}$, we have
 (\ref{tau-aaa}) and (\ref{tau-bbb}) hold.}
\end{equation}
Since $\tau_n < \tau_{n-1}$,
by the linesearch procedure in Step 2, we have $\tau_n = \varphi \tau_{n-1} \mu^j$
with $j \ge 1$ and $\sqrt{\beta \varphi \tau_{n-1} \mu^{j-1}} L > \sigma \sqrt{\psi/\tau_{n-1}}$,
which is equivalent to
\begin{equation}\label{tau-aaa}
\tau_{n-1} > \underline{\tau} \mu^{-(j-1)/2}.
\end{equation}
On the other hand, by \eqref{tau-aaa} and $j \ge 1$, we have
\begin{eqnarray}\label{tau-bbb}
1 + \ell(\tau_n) &= & 1 + \lfloor \log_{\varphi} (\underline{\tau}/\tau_n) \rfloor
= 1 + \lfloor \log_{\varphi} (\underline{\tau}/( \varphi \tau_{n-1} \mu^j)) \rfloor \nonumber \\
 &=& \lfloor \log_{\varphi} (\underline{\tau}/( \tau_{n-1} \mu^j)) \rfloor
 \le \left\lfloor \frac{j+1}{2}\log_{\varphi} (1/\mu) \right\rfloor \nonumber \\
&\le& j \log_{\varphi} (1/\mu).
\end{eqnarray}

Let $\mathcal{Z}^+ := \{1,2,3,\ldots\}$ be the set of positive integers.
Given two integers $i_1 \le i_2$, let
interval $[i_1, i_2] := \{i \in \mathcal{Z}^+: i_1 \le i \le i_2\}$ and
interval $[i_1, \infty) :=  \{i \in \mathcal{Z}^+: i_1 \le i < \infty \}$.
Then, based on properties \eqref{property-A},
 \eqref{property-B} and the assumption $\tau_1 \ge \underline{\tau}$, there exist a set of positive integers
$\mathcal{K} := \cup_{i=1}^{|\mathcal{K}|}\{k_i\} \subseteq \mathcal{Z}^+$ and an associated integer set
 $\mathcal{M} := \cup_{i=1}^{|\mathcal{K}| - 1} \{m_i\} \subseteq \mathcal{Z}^+$,
 where $|\mathcal{K}|\geq 1$ denotes the cardinality of $\mathcal{K}$ that is either a finite number or infinity,
 such that they partition $\mathcal{Z}^+$, i.e., $\mathcal{Z}^+ = \cup_{i=1}^{\infty}[k_i, k_{i+1}-1]$
 if $|\mathcal{K}| =\infty$ or
 $\mathcal{Z}^+ = \cup_{i=1}^{|\mathcal{K}|-1}[k_i, k_{i+1}-1] \cup [k_{|\mathcal{K}|}, \infty)$
 if $|\mathcal{K}| < \infty$,
  and the following properties hold:
 \begin{enumerate}
\item[(a)] $k_1 = 1$ and $k_i < m_i < k_{i+1}$ for all $i$;
\item[(b)] $\tau_k \ge \underline{\tau}$ for all $k \in [k_i, m_i-1]$ and
 $\tau_k < \underline{\tau}$ for all $k \in [m_i, k_{i+1}-1]$, see the diagram below
 \ben
\overbrace{\ldots,~~~k_i-1,} ~~ \overbrace{ k_i,~~~\ldots,~~~m_i-1,}^{t_k\geq\underline{\tau}, \,\, (m_i-k_i) \text{~times}} ~~ \overbrace{m_i,~~~\ldots,~~~k_{i+1}-1,}^{t_k<\underline{\tau}, \,\, (k_{i+1}-m_i) \text{~times}} ~~ \overbrace{k_{i+1},~~\ldots}
\een
\item[(c)] If $|\mathcal{K}| < \infty$,  $ \tau_k \ge \underline{\tau}$ for all
$k \ge k_{|\mathcal{K}|}$; Otherwise,  $|\mathcal{K}| = \infty$ and
$\mathcal{Z}^+  = \cup_{i=1}^{\infty} [k_i, k_{i+1}-1]$;
\item[(d)] $ \tau_{k_i} < \varphi \underline{\tau}$ by property \eqref{property-A} for all $k_i \in
\mathcal{K}\setminus \{k_1\}$;
\item[(e)] $\tau_{m_i-1} >  \underline{\tau} \mu^{-(j-1)/2}$ by \eqref{tau-aaa}
and $k_{i+1}-m_i = \ell(\tau_{m_i}) + 1 \le j \log_{\varphi}(1/\mu)$ by \eqref{tau-bbb}
for all $m_i \in \mathcal{M}$ and some $j \ge 1$ depending on  $m_i$.
 \end{enumerate}
Now, we consider any interval $[k_i, k_{i+1}-1] = [k_i, m_i-1] \cup [m_i,  k_{i+1}-1]$.
Let $j\geq 1$ be the integer associated with $m_i$ such that property (e) holds.
Since $\tau_{k+1} \le \varphi \tau_k$ for all $k$, we have
\ben
\tau_{m_i-1} < \tau_{k_i} \varphi^{m_i-1-k_i}.
\een
Then, by properties (d) and (e), for $k_i \ne k_1=1$, we have
 $ \tau_{k_i} < \varphi \underline{\tau}$ and
 $\tau_{m_i-1} \ge \underline{\tau} \mu^{-(j-1)/2}$, which
 together with the above inequality gives
 $\underline{\tau} \mu^{-(j-1)/2} \leq \tau_{m_i-1} < \tau_{k_i} \varphi^{m_i-1-k_i} <
 \underline{\tau}  \varphi^{m_i-k_i}$, which is equivalent to
 \ben
 (j-1) \log_{\varphi} (1/\sqrt{\mu}) < m_i - k_i.
 \een
For $k_i = k_1 =1$, we have $\underline{\tau} \mu^{-(j-1)/2} < \tau_{1} \varphi^{m_i-2}$, or
$(j-1) \log_{\varphi} (1/\sqrt{\mu}) \le \log_{\varphi} \big(\tau_1/\underline{\tau} \big) +  m_i - 2$.
So, there exists an integer constant $\bar{j} \ge 1$, which does not depend on either $k_i$ or $m_i$,
such that
\begin{equation}\label{X-CCC}
\mbox{$m_i - k_i  \geq  \frac{j}{4} \log_{\varphi} (1/\mu)$ \; whenever \; $j \ge \bar{j}$.}
 \end{equation}
Next, we show that there exists a constant $\bar{c} >0$, which does not depend on $m_i$
or the interval $[k_i, k_{i+1}-1]$, such that
\begin{equation}\label{X-interval-ratio}
m_i - k_i \ge \bar{c}\, (k_{i+1}-m_i).
\end{equation}
Noticing that by property (e), we have $ k_{i+1} - m_i\le j \log_{\varphi}(1/\mu)$.
Hence, it follows from \eqref{X-CCC} that if $j \ge \bar{j}$ we have
\begin{equation*}
m_i - k_i \ge \frac{ \log_{\varphi} (1/\mu)}{4  \log_{\varphi}(1/\mu)}  (k_{i+1}-m_i) =
\frac{1}{4} (k_{i+1}-m_i).
\end{equation*}
On the other hand, if $ 1\le j < \bar{j}$,  we have from  $ k_{i+1} - m_i\le j \log_{\varphi}(1/\mu)$ that
\ben
 m_i - k_i \geq 1 >  j/ \bar{j}  \ge   \frac{k_{i+1}-m_i}{\bar{j} \log_{\varphi}(1/\mu)}.
\een
Hence, \eqref{X-interval-ratio} holds with
$\bar{c} = \min\{1/( \bar{j} \log_{\varphi}(1/\mu) ), 1/4\} > 0$.

Apparently, if follows from \eqref{X-interval-ratio} that
$m_i - k_i \ge \hat{c}\, (k_{i+1}-k_i)$ with $\hat{c} := \bar{c}/(1+\bar{c}) > 0$.
If $|\mathcal{K}| = \infty$, given any $N \ge 1$, it follows from property (c) that
$N \in [k_i, k_{i+1}-1]$ for certain $i \ge 1$. Hence, by the definition of
$\mathcal{K}_N = \{1 \le n \le N: \tau_n \ge \underline{\tau}\}$, we have
$\mathcal{K}_N = \cup_{j=1}^i \left( \mathcal{K}_N \cap [k_j, k_{j+1}-1] \right)$.
Then, it follows from \eqref{X-interval-ratio} and property (b) that
\begin{eqnarray}\label{X-ratio-1}
|\mathcal{K}_N | &=& \sum_{j=1}^i \left| \mathcal{K}_N \cap [k_j, k_{j+1}-1] \right|
                 \ge  \sum_{j=1}^i \left| \mathcal{K}_N \cap [k_j, m_j-1] \right| \nonumber \\
                & = & \min\{N, m_i \} - k_i  + \sum_{j=1}^{i-1} \left|[k_j, m_j-1] \right|
                 = \min\{N, m_i \} - k_i  + \sum_{j=1}^{i-1}  (m_j-k_j) \nonumber \\
                 &\ge & \hat{c} (N-k_i) +  \hat{c }\sum_{j=1}^{i-1}  (k_{j+1}-k_j)
                 = \hat{c} N,
\end{eqnarray}
where $\hat{c} = \bar{c}/(1+\bar{c})$.
If $|\mathcal{K}| < \infty$, by property (c), for all $n \ge k_{|\mathcal{K}|}$
we have $\tau_n \ge  \underline{\tau}$. This property together with \eqref{X-ratio-1} implies
$|\mathcal{K}_N | \ge \hat{c} N$ for any $N \ge 1$.
\end{proof}

\section{Proof of Lemma \ref{lem-beta-cnn}}\label{proof:lem-beta-cnn}
\begin{proof}
(i) This conclusion follows almost from an identical proof of conclusion (i) in Lemma~\ref{lem_bound}
except by replacing $\beta$ by $\beta_n$ and setting $\sigma =1$.

(ii) Let $h(\tau) := 1+ \frac{(\psi-\varphi)\gamma_g\tau}{\psi+\varphi\gamma_g\tau}$.
Since $h(\tau)$ is strictly increasing with respect to $\tau>0$, we have that
$1 < 1+\gamma_g \omega_{n}\tau_{n-1} = h(\tau_{n-1}) < \varsigma :=\psi/\varphi$.
So, by \eqref{beta-gstrong}, $\beta_{n-1} < \beta_{n} = \beta_{n-1} h(\tau_{n-1}) < \varsigma \beta_{n-1}$.
Therefore, for any $\sqrt{\beta_{n-1}} \tau_{n-1} \le 1/L$,  we have
$\sqrt{\beta_{n}} \tau_{n-1} \le \sqrt{\varsigma \beta_{n-1}} \tau_{n-1}
\le \sqrt{\varsigma}/L = \frac{1}{L} \sqrt{\frac{\psi}{\varphi}}$,
which by the linesearch procedure in Step 2 implies that
the initial trial $\tau_{n} = \varphi \tau_{n-1}$ will satisfy \eqref{y-ls-gstrong} and be accepted
by the linesearch.
Hence, analogous to property \eqref{property-A},
 we have the following property:
\begin{eqnarray} \label{property-A-strong}
& \mbox{ If $\sqrt{\beta_n} \tau_n < \frac{1}{L}$, then there exists an integer $\ell_n \ge 0$
such that $k' :=n+\ell_n$ satisfies: }  &  \nonumber\\
&  \mbox{$\sqrt{\beta_k} \tau_k < \frac{1}{L}$ and $\tau_{k+1} = \varphi \tau_k$
for all $k = n, n+1, \ldots, k'$, and $\sqrt{\beta_{k'+1}} \tau_{k'+1} \ge \frac{1}{L}$
}, &
\end{eqnarray}
which, by $ \beta_n < \beta_{n+1} < \varsigma \beta_n $ and $\tau_{n+1} \le \varphi \tau_n$
for all $n \ge 1$, also implies
\begin{equation}\label{XXX}
\ell_n \leq \vartheta \big(\sqrt{\beta_n} \tau_n \big) \quad \mbox{and} \quad
\sqrt{\beta_{k'+1}} \tau_{k'+1} < \sqrt{\varsigma \beta_{k'}}
\varphi \tau_{k'} = \sqrt{\psi \varphi} \sqrt{\beta_{k'}} \tau_{k'}
< \theta,
\end{equation}
where $\vartheta(t) := \lfloor \log_{\varphi}(1/(L t)) \rfloor$ for $t < 1/L$
and $\theta := \sqrt{\psi \varphi}/L$.

Analogous to property \eqref{property-B}, we show the following property:
\begin{equation} \label{property-B-strong}
\mbox{For any $\sqrt{\beta_{n-1}} \tau_{n-1} \ge \frac{1}{L}$ and $\sqrt{\beta_n} \tau_{n} < \frac{1}{L}$,
we have  (\ref{AAA}) and (\ref{BBB}) hold. }
\end{equation}
Since $\beta_n > \beta_{n-1}$ and $\sqrt{\beta_n} \tau_{n} < \sqrt{\beta_{n-1}} \tau_{n-1}$, by the linesearch procedure in
Step 2, we have $\tau_n = \varphi \tau_{n-1} \mu^j$ with $j \ge 1$ and
$\sqrt{ \beta_n \varphi  \tau_{n-1} \mu^{j-1}} L > \sqrt{\psi/\tau_{n-1}}$, which gives
\begin{equation}\label{AAA}
\sqrt{\beta_n} \tau_{n-1} \ge \sqrt{\frac{\psi}{ \varphi}} \frac{\mu^{-{j-1 \over 2}}}{L}
\quad \mbox{and} \quad \sqrt{\beta_{n-1}} \tau_{n-1} > \sqrt{\beta_n/\varsigma} \tau_{n-1} \ge
\frac{\mu^{-{j-1 \over 2}}}{L }.
\end{equation}
It follows from \eqref{XXX},
$\psi > \varphi > 1$, $j \ge 1$ and  \eqref{AAA} that
\begin{eqnarray}\label{BBB}
1+ \ell_n &\leq& 1+  \vartheta (\sqrt{\beta_n} \tau_n )
  = 1 + \left\lfloor \log_{\varphi}\big(1/(L \sqrt{\beta_n} \tau_n)\big) \right\rfloor\nonumber \\
& = & 1+\left\lfloor \log_{\varphi}\big(1/(L \sqrt{\beta_n} \varphi \tau_{n-1} \mu^j)\big) \right\rfloor
= \left\lfloor \log_{\varphi}(1/(L \sqrt{\beta_n} \tau_{n-1} \mu^j)) \right\rfloor \nonumber \\
&\le & \left\lfloor \frac{j+1}{2}\log_{\varphi}(1/\mu) +
 \frac{1}{2}\log_{\varphi} (\varphi/\psi) \right\rfloor \nonumber \\
 &\le & j  \log_{\varphi}(1/\mu).
\end{eqnarray}

Let $\mathcal{Z}^+ := \{1,2,3,\ldots\}$ be the set of positive integers. Given two
integers $i_1 \le i_2$, let
interval $[i_1, i_2] := \{i \in \mathcal{Z}^+: i_1 \le i \le i_2\}$ and
interval $[i_1, \infty) :=  \{i \in \mathcal{Z}^+: i_1 \le i < \infty \}$.
To show this lemma, without losing of generality, by property \eqref{property-A-strong}, we can
simply assume $\sqrt{\beta_1} \tau_1 \geq 1/L$.
Then, based on properties \eqref{property-A-strong} and
 \eqref{property-B-strong}, there exist a set of positive integers
$\mathcal{K} := \cup_{i=1}^{|\mathcal{K}|}\{k_i\} \subseteq \mathcal{Z}^+$ and an associated integer set
 $\mathcal{M} := \cup_{i=1}^{|\mathcal{K}| - 1} \{m_i\} \subseteq \mathcal{Z}^+$,
 where $|\mathcal{K}|\geq 1$ denotes the cardinality of $\mathcal{K}$ that is either a finite number or infinity,
 such that they partition $\mathcal{Z}^+$, i.e., $\mathcal{Z}^+ = \cup_{i=1}^{\infty}[k_i, k_{i+1}-1]$ if $|\mathcal{K}| =\infty$ or
 $\mathcal{Z}^+ = \cup_{i=1}^{|\mathcal{K}|-1}[k_i, k_{i+1}-1] \cup [k_{|\mathcal{K}|}, \infty)$ if $|\mathcal{K}| < \infty$,
  and the following properties hold:
 \begin{enumerate}
\item[(a)] $k_1 = 1$ and $k_i < m_i < k_{i+1}$ for all $i$;
\item[(b)] $\sqrt{\beta_{k}} \tau_k \ge 1/L$ for all $k \in [k_i, m_i-1]$ and
 $\sqrt{\beta_{k}} \tau_k < 1/L$ for all $k \in [m_i, k_{i+1}-1]$;
\item[(c)] If $|\mathcal{K}| < \infty$,  $\sqrt{\beta_{k}} \tau_k \ge 1/L$ for all
$k \ge k_{|\mathcal{K}|}$; Otherwise,  $|\mathcal{K}| = \infty$ and
$\mathcal{Z}^+  = \cup_{i=1}^{\infty} [k_i, k_{i+1}-1]$;
\item[(d)] $\sqrt{\beta_{k_i}} \tau_{k_i} < \theta$ by \eqref{XXX} for all $k_i \in
\mathcal{K}\setminus \{k_1\}$, where $\theta =  \sqrt{\psi \varphi}/L$;
\item[(e)] $\sqrt{\beta_{m_i-1}} \tau_{m_i-1} > \mu^{-(j-1)/2}/L$ by \eqref{AAA}
and $k_{i+1}-m_i = \ell_{m_i} + 1 \le j \log_{\varphi}(1/\mu)$ by \eqref{BBB}
for all $m_i \in \mathcal{M}$ and some $j \ge 1$ depending on  $m_i$, where $\ell_{m_i}$ is defined in property \eqref{property-A-strong} associated with $\sqrt{\beta_{m_i}}\tau_{m_i}$.
 \end{enumerate}
Now, we consider any interval $[k_i, k_{i+1}-1] = [k_i, m_i-1] \cup [m_i,  k_{i+1}-1]$.
Let $j\geq 1$ be the integer associated with $m_i$ such that property (e) holds.
Since $\sqrt{\beta_{k+1}} \tau_{k+1} <  \sqrt{\varsigma \beta_k} \varphi \tau_k =
 \sqrt{\psi \varphi} \sqrt{\beta_k} \tau_k$ for all $k$, we have
\ben
\sqrt{\beta_{m_i-1}} \tau_{m_i-1} < \sqrt{\beta_{k_i}} \tau_{k_i} \rho^{(m_i-1-k_i)/2},
\een
where $\rho := \psi \varphi >1$.
By properties (d) and (e), for $k_i \ne k_1=1$, we have  $\sqrt{\beta_{k_i}} \tau_{k_i} < \theta$ and
 $\sqrt{\beta_{m_i-1}} \tau_{m_i-1} \ge \mu^{-(j-1)/2}/L$, which
 together with the above inequality gives
 \ben
 \mu^{-(j-1)/2}/L < \theta  \rho^{(m_i-1-k_i)/2} \quad \Longleftrightarrow \quad
 j \log_{\rho} (1/\mu) &<& \log_{\rho} \big((\theta L)^2/(\mu \rho)\big) + m_i - k_i\\
 &=&\log_{\rho} \big(1/\mu\big) + m_i - k_i.
 \een
For $k_i = k_1 =1$, we have
$j \log_{\rho} (1/\mu) \le \log_{\rho} \big((\theta_1 L)^2/(\mu \rho)\big) +  m_i - 1$, where
$\theta_1 := \sqrt{\beta_1} \tau_1$.
So, there exists an integer constant $\bar{j} \ge 1$, which does not depend on either $k_i$ or $m_i$,
such that
\begin{equation}\label{CCC}
\mbox{$m_i - k_i  \geq  \frac{j}{2} \log_{\rho} (1/\mu)$ \; whenever \; $j \ge \bar{j}$.}
 \end{equation}
Next, we show that there exists a constant $\bar{c} >0$, which does not depend on $m_i$, such that
\begin{equation}\label{interval-ratio}
m_i - k_i \ge \bar{c}\, (k_{i+1}-m_i + 1).
\end{equation}
Noticing that by property (e), we have
\begin{equation}\label{DDD}
k_{i+1} - m_i +1 \le j \log_{\varphi}(1/\mu) +1 \le j \,(\log_{\varphi}(1/\mu) +1).
\end{equation}
Hence, it follows from \eqref{CCC} and \eqref{DDD} that whenever $j \ge \bar{j}$ we have
\begin{equation*}
m_i - k_i \ge \frac{ \log_{\rho} (1/\mu)}{2 (\log_{\varphi}(1/\mu) +1)}  (k_{i+1}-m_i + 1).
\end{equation*}
On the other hand, if $ 1\le j < \bar{j}$, then by \eqref{DDD} we have
\ben
 m_i - k_i \geq 1 >  j/ \bar{j}  \ge   \frac{k_{i+1}-m_i+1}{\bar{j} \big(\log_{\varphi}(1/\mu)+1\big)}.
\een
Hence, \eqref{interval-ratio} holds with
$\bar{c} = \min\{1/(\bar{j} (\log_{\varphi}(1/\mu))+1)), \log_{\rho} (1/\mu)/(2 (\log_{\varphi}(1/\mu)+1))\} > 0$.

Now, by \eqref{beta-gstrong} we have $\beta_{n+1} = \beta_n h(\tau_n) = \beta_n (1+ t_n)$, where
$t_n = \frac{(\psi-\varphi)\gamma_g\tau_n}{\psi+\varphi\gamma_g\tau_n}$.
 Then, we have
\begin{eqnarray}\label{beta-diff}
\sqrt{\beta_{n+1}} - \sqrt{\beta_n} &= & \frac{\beta_{n+1} - \beta_n}{\sqrt{\beta_{n+1}} + \sqrt{\beta_n}}
= \frac{t_n \beta_n}{\sqrt{\beta_{n+1}} + \sqrt{\beta_n}} \nonumber \\
&\ge & \frac{t_n \beta_n}{(\sqrt{\varsigma} +1) \sqrt{\beta_n}} \ge \tilde{c}
\min \{\tau_n, 1\} \sqrt{\beta_n},
\end{eqnarray}
where $\tilde{c} >0$ is some constant. Consider any $k \in [k_i, k_{i+1}-1] = [k_i, m_i-1] \cup [m_i,  k_{i+1}-1]$.
If $k \in [m_i, k_{i+1}-1]$,
by \eqref{interval-ratio} we have
\begin{equation}\label{ratio-2}
\frac{m_i - k_i}{k_{i+1} - k_i+1} = \frac{m_i - k_i}{(m_i - k_i) + (k_{i+1} - m_i+1)}
\ge \frac{1}{1+1/\bar{c}} = \frac{\bar{c}}{1+\bar{c}},
\end{equation}
and it thus follows from \eqref{beta-diff},
  property (b),   $\beta_n \ge \beta_0 >0$ and \eqref{ratio-2} that
 \begin{eqnarray}
 \sqrt{\beta_k} - \sqrt{\beta_{k_i}} &=& \sum_{n=k_i}^{k-1}
 \Big(\sqrt{\beta_{n+1}} - \sqrt{\beta_n}\Big)
 \ge \tilde{c} \sum_{n=k_i}^{k-1}  \min \{\tau_n, 1\} \sqrt{\beta_n} \nonumber \\
 &\ge& \tilde{c} \sum_{n=k_i}^{m_i-1}  \min \{\tau_n, 1\} \sqrt{\beta_n}
 \ge \tilde{c} \sum_{n=k_i}^{m_i-1} \min \{1/L, \sqrt{\beta_0}\} \nonumber \\
 & = &
  \tilde{c}  (m_i - k_i) \min \{1/L, \sqrt{\beta_0}\}
  \ge   \hat{c} (k_{i+1} - k_i + 1) \ge \hat{c}\, (k - k_i + 1), \label{beta-diff-n}
 \end{eqnarray}
 where $\hat{c} := \tilde{c} \, \bar{c}  \min \{1/L, \sqrt{\beta_0}\} /(1+\bar{c}) > 0$,
 $\bar{c}$ and $\tilde{c}$ are constants given in \eqref{interval-ratio} and \eqref{beta-diff},
 respectively.
 On the other hand,  if $k \in [k_i, m_i-1]$, similar to \eqref{beta-diff-n} we can show
 \begin{equation}\label{beta-diff-n1}
 \sqrt{\beta_k} - \sqrt{\beta_{k_i}} \ge \tilde{c} \min \{1/L, \sqrt{\beta_0}\} (k-k_i) \ge \hat{c}\, (k-k_i).
 \end{equation}
% So, for any $k \in [k_i, k_{i+1}-1]$, we have
%  \comm{previously the RHS is $(k-k_i+1)$, but it is not true for $k=k_i$.}
% \begin{equation}\label{beta-diff-n}
% \sqrt{\beta_k} - \sqrt{\beta_{k_i}} \ge \hat{c}\, \red{(k-k_i)}.
% \end{equation}
 If $|\mathcal{K}| = \infty$, given any $n \ge 1$, it follows from
  property (c) that $n \in [k_i, k_{i+1} - 1]$ for certain $i \ge 1$.
  Hence, by \eqref{beta-diff-n}, \eqref{beta-diff-n1} and $0 < \beta_j < \beta_{j+1}$ for any $j \ge 1$, we have
  \begin{eqnarray}\label{beta-rate}
  \sqrt{\beta_n} &=& \sqrt{\beta_n} - \sqrt{\beta_{k_i}} +
   \sum_{j=1}^{i-1} (\sqrt{\beta_{k_{j+1}}} - \sqrt{\beta_{k_j}}) + \sqrt{\beta_{k_1}} \nonumber \\
   &=& \sqrt{\beta_n} - \sqrt{\beta_{k_i}} +
   \sum_{j=1}^{i-1} \left[\Big(\sqrt{\beta_{k_{j+1}-1}} - \sqrt{\beta_{k_j}}\Big)
   + \Big(\sqrt{\beta_{k_{j+1}}} - \sqrt{\beta_{k_{j+1}-1}}\Big)\right]  + \sqrt{\beta_{k_1}} \nonumber\\
  & \ge & \sqrt{\beta_n} - \sqrt{\beta_{k_i}} +
   \sum_{j=1}^{i-1} \Big(\sqrt{\beta_{k_{j+1}-1}} - \sqrt{\beta_{k_j}}\Big)  \nonumber \\
   & \ge & \hat{c} \Big[ (n-k_i) + \sum_{j=1}^{i-1} (k_{j+1} - k_j)  \Big] = \hat{c}\,   (n - 1).
  \end{eqnarray}
 If $|\mathcal{K}| < \infty$, by property (c),  for any $k \ge k_{|\mathcal{K}|}$ we have
 $\sqrt{\beta_{k}} \tau_k \ge 1/L$ and thus
 \ben
 \sqrt{\beta_k} - \sqrt{\beta_{k_{|\mathcal{K}|}}}
 %\ge \tilde{c} \min \{1/L, \sqrt{\beta_0}\} (k-k_{|\mathcal{K}|}+1)
 \ge \hat{c} (k-k_{|\mathcal{K}|}).
 \een
  This together with \eqref{beta-diff-n} %, which holds for any $k \in [k_i, k_{i+1}-1]$ and $i\in [1, |\mathcal{K}|-1]$,
  and property (c) also implies \eqref{beta-rate} holds
  for any $n \ge 1$. Then, conclusion (ii) follows from \eqref{beta-rate}.

(iii). Note that for any $\sqrt{\beta_p} \tau_p \ge \sqrt{\beta_q} \tau_q $ with $p \le q$, we have from $\beta_p \le \beta_q$ that $\tau_p \ge \tau_q$. So, based on property (b) in (ii) and \eqref{interval-ratio}, we have the property:
 \ben
 \mbox{(b') $\tau_p \ge \tau_q$ for all $p \in [k_i, m_i-1]$, $q \in [m_i, k_{i+1}-1]$
 and therefore $\sum_{n=k_i}^{m_i-1} \tau_n \ge \bar{c} \sum_{n=m_i}^{k_{i+1}-1} \tau_n$},
 \een
 where $\bar{c} >0$ is the constant in \eqref{interval-ratio}.

Let $\mathcal{K} \subset \mathcal{Z}^+$ be the set given in the proof of (ii).
If $|\mathcal{K}| = \infty$, then for any $N \ge 1$ it follows from property (c) that $N \in [k_j, k_{j+1} - 1]$ for certain $j \ge  1=k_1$. Hence, we have from property  (b') that
 \begin{eqnarray*}
  \sum_{n=1}^N \tau_n
  &=& \sum_{n=k_j}^N \tau_n + \sum_{s=1}^{j-1}    \sum_{n=k_s}^{k_{s+1}-1} \tau_n
   =  \sum_{n=k_j}^N \tau_n + \sum_{s=1}^{j-1}  \left( \sum_{n=k_s}^{m_s-1} \tau_n  + \sum_{n=m_s}^{k_{s+1}-1} \tau_n \right) \\
  &\le & (1+1/\bar{c}) \sum_{n=k_j}^{\min\{N, m_j-1\}} \tau_n +   \sum_{s=i}^{j-1} \left(\sum_{n=k_s}^{m_s-1} \tau_n +  1/\bar{c} \sum_{n=k_s}^{m_s-1} \tau_n \right) \\
  &=& (1+1/\bar{c}) \left( \sum_{n=k_j}^{\min\{N, m_j-1\}} \tau_n  +   \sum_{s=1}^{j-1} \sum_{n=k_s}^{m_s-1} \tau_n \right)
   = (1+1/\bar{c}) \sum_{n\in \mathcal{S}_N} \tau_n.
 \end{eqnarray*}
 When $|\mathcal{K}| < \infty$, we can also similarly prove that
  $\sum_{n=1}^N \tau_n \le (1+1/\bar{c}) \sum_{n\in \mathcal{S}_N} \tau_n$ for any $N\geq 1$, because $\sqrt{\beta_{n}} \tau_n \ge 1/L$ for all $n \geq k_{|\mathcal{K}|}$.
The proof is completed by letting $\tilde{c} = 1 + 1/\bar{c}$.
\end{proof}

\end{appendix}

\end{document}